\documentclass[12pt]{amsart}

\usepackage[top=1in, bottom=1in, left=1in, right=1in]{geometry}
\usepackage{times}
\usepackage{amssymb,amsmath}
\usepackage{amsthm}
\usepackage{enumitem}

\makeatletter
\renewcommand\subsection{\@startsection{subsection}{2}%
	\z@{1\linespacing\@plus1\linespacing}{\linespacing}%
	{\normalfont\bfseries\centering}}
\makeatother

\setlist[1]{itemsep=0.5em, topsep=0.5em}

\usepackage[dvipsnames]{xcolor}
\usepackage[colorlinks=true, linkcolor=Cerulean, citecolor=ForestGreen, urlcolor=Plum]{hyperref}   %% linkable cross-refs

\usepackage{cleveref}
\usepackage{bm}
\renewcommand{\le}{\leqslant}
\renewcommand{\leq}{\leqslant}
\renewcommand{\ge}{\geqslant}
\renewcommand{\geq}{\geqslant}

\numberwithin{equation}{section}

\theoremstyle{plain}
\newtheorem{thm}{Theorem}[section]
\newtheorem*{thm*}{Theorem}
\newtheorem{cor}[thm]{Corollary}
\newtheorem{lem}[thm]{Lemma}
\newtheorem{prop}[thm]{Proposition}

\theoremstyle{definition}
\newtheorem{rem}[thm]{Remark}
\newtheorem*{rem*}{Remark}

\renewcommand{\l}{\ell}
\newcommand{\N}{\mathbb{N}}
\newcommand{\Z}{\mathbb{Z}}
\newcommand{\Q}{\mathbb{Q}}
\newcommand{\R}{\mathbb{R}}

\newcommand{\E}{\mathbb{E}}
\renewcommand{\P}{\mathbb{P}}

\newcommand{\CB}{\mathcal{B}}
\newcommand{\CC}{\mathcal{C}}

\newcommand{\CH}{\mathcal{H}}

\newcommand{\CM}{\mathcal{M}}
\newcommand{\CN}{\mathcal{N}}

\newcommand{\CS}{\mathcal{S}}

\newcommand{\CV}{\mathcal{V}}

\newcommand{\supp}{\operatorname{supp}}

\renewcommand{\phi}{\varphi}

\renewcommand{\Re}{\mathrm{Re}}
\renewcommand{\Im}{\mathrm{Im}}

\usepackage{mathrsfs}

\begin{document}
	\title[random multiplicative functions with periodic weights]{Random multiplicative functions with periodic weights}
	\author[J. Schlitt]{Jeremy Schlitt}
	\address{D\'epartement de math\'ematiques et de statistique\\
		Universit\'e de Montr\'eal\\
		CP 6128 succ. Centre-Ville\\
		Montr\'eal, QC H3C 3J7\\
		Canada}
	\email{jeremy.schlitt@umontreal.ca}
	\date{\today}
	\begin{abstract}
		Given a Steinhaus random multiplicative function $f$, a $1$-periodic function of bounded variation $g$, and an irrational number $\alpha$, we study the distribution of $\sum_{n=1}^N f(n) g(\alpha n)$. We determine a necessary and sufficient condition for these sums, normalized by their standard deviation, to converge to the standard complex Gaussian distribution. On the other hand, we show that if we restrict the summands to have all their prime factors  $>z$, with $z$ tending to infinity arbitrarily slowly, then a central limit theorem always holds for such sums.
	\end{abstract}
	\maketitle
	\section{Introduction}
	The ensemble of Steinhaus random multiplicative functions is defined as follows: Let $\{f(p)\}_{p \text{ prime}}$ be a sequence of independent random variables which are identically uniformly distributed on the complex unit circle. We then extend the definition of $f$ to all natural numbers by complete multiplicativity:
	$$f(n) := \prod_{p^k \parallel n}f(p)^k.$$
	The sequence $\{f(n)\}_{n \in \N}$ is a fascinating example of a sequence of \textit{non-independent} random variables, since there are relations between the random variables generated by the factorization of integers into prime numbers. Random multiplicative functions, first introduced by Wintner~\cite{wintner} in 1944, have seen a resurgence since the breakthrough work of Harper \cite{harper2020moments}. One of Harper's main results is that
	\begin{equation}\label{eq:harper better sqrt}
		\E\bigg|\frac{1}{\sqrt{N}}\sum_{n \leq N} f(n)\bigg|  = o(1),
	\end{equation}
	which implies that Steinhaus random multiplicative functions \textit{do not} obey a central limit theorem. In light of this, one may ask if such a limit theorem holds for an appropriately twisted version of the sum \eqref{eq:harper better sqrt}. In this direction, Soundararajan and Xu \cite{sound-xu-CLT} recently gave a rather general criterion for which weights $a_n$ make the appropriately normalized sum of random variables
	$$\sum_{n \leq N}a_nf(n)$$
	converge in distribution to a standard complex Gaussian random variable. See \Cref{sec:MDS} below for a precise statement of their general result. A particular application of their work is the following:
	\begin{thm}[{Soundararajan \& Xu, \cite[Corollary 1.6]{sound-xu-CLT}}] \label{thm: sound-xu 1.6}
		Let $\theta$ denote an irrational number such that for some positive constant $C = C(\theta)$ we have
		\begin{equation*}
			\|q\theta\|:= \min_{n \in \Z}|q \theta -n |\geq C\exp(-|q|^{1/50})\quad \forall q \in \Z \setminus \{0\}.
		\end{equation*}
		Then as $N\to \infty$,
		$$\frac{1}{\sqrt{N}}\sum_{n \leq N}e(n \theta)f(n)$$
		converges in distribution to a standard complex normal random variable with mean $0$ and variance $1$.
	\end{thm}
	It is a guiding heuristic in analytic number theory that multiplicative and additive structures are incommensurate. \Cref{thm: sound-xu 1.6} indicates that introducing additive character weights into a sum of random multiplicative functions effectively ``kills'' the multiplicative dependencies, making the sum behave like a sum of independent random variables and thus yielding a central limit theorem. 
	
	Since complex exponentials form an orthonormal basis for $L^2(\R/\Z),$ it is natural to ask whether a central limit theorem holds for sums of random multiplicative functions weighted by $1$-periodic functions. In the present paper, we give a complete characterization of the periodic weights $g(\theta n)$ that lead to a central limit theorem, under some minor assumptions on $g$. 
	
	Before stating our results, some preliminary setup is required. Throughout the paper, $\theta$ is a fixed irrational number, and we assume there exists a constant $C = C(\theta)$ such that
	\begin{equation}\label{eq:sound-xu badness approx}
		\|q\theta\|:= \min_{n \in \Z}|q \theta -n |\geq C\exp(-|q|^{1/127}) \quad \forall q \in \Z \setminus \{0\}.
	\end{equation}
	\begin{rem}
		The condition \eqref{eq:sound-xu badness approx} is extremely nonrestrictive. The set of $\theta \in [0,1]$ satisfying this condition has Lebesgue measure $1$. In fact, one can even show that the exceptions form a set of Hausdorff dimension $0$. For these reasons, little attempt was made to improve the diophantine exponent $1/127$ to something larger. 
	\end{rem}
	
	All implicit constants may depend on $\theta$ even if not indicated. We denote the family of all functions of bounded variation by $\CB\CV[\R/\Z]$, and the mean-zero functions of bounded variation by $\CB\CV_0[\R/\Z]$. See the notation section at the end of the introduction for precise definitions of these families.
	
	For a given $g \in \CB\CV[\R/\Z]$, consider the random variable
	\begin{equation}\label{eq:def S_N}
		S_{g}(N) := \frac{1}{\sqrt{V_{g}(N)}}\sum_{n \leq N} f(n) g\big(\theta n\big),
	\end{equation}
	where $f$ is sampled from the ensemble of Steinhaus random multiplicative functions, and $V_{g}(N)$ is chosen so that $\E[|S_g(N)|^2] = 1$. The variance $V_g$ can be computed explicitly by using orthogonality relations for Steinhaus random multiplicative functions:
	$$V_{g}(N) := \E\Big[\Big|\sum_{n \leq N} f(n)g(\theta n)\Big|^2\Big] = \sum_{n \leq N}\Big|g\big(\theta n\big)\Big|^2.$$
	
	It is a recurring phenomenon in the study of random multiplicative functions that the limiting distribution is heavily dictated by the behavior of the fourth moment. This principle has been established across several distinct regimes, including short intervals (Chatterjee and Soundararajan \cite{chatterjee2012random}), integers with restricted prime factors (Harper \cite{harper2013limit}), polynomial evaluations (Klurman, Shkredov, and Xu \cite{klurman2023random}, and Chinis and Shala \cite{chinis2025random}). The aforementioned work of Soundararajan and Xu demonstrates that negligible `off-diagonal' fourth-moment contributions are sufficient for $S_g(N)$ to have a Gaussian limiting distribution. In this paper, we prove that for periodic weights, this condition is in fact necessary. 
	
	We are now ready to state our first theorem, which gives a precise characterization of when $S_g(N)$ converges in distribution to a complex normal random variable.
	\begin{thm}\label{thm:A}
		Fix $\theta$ to be an irrational number satisfying \eqref{eq:sound-xu badness approx}. Let $g \in \CB\CV[\R/\Z]$ be a non-constant function. Then the following two statements are equivalent.
		\begin{enumerate}
			\item  $\{S_{g}(N)\}_{N \geq 1}$ converges in distribution to a complex normal distribution with mean $0$ and variance $1-|\hat{g}(0)|^2/\|g\|_2^2$ as $N \to \infty$
			\item For all positive integers $a,b$, such that $a \ne b$ and $(a,b) = 1$, one has $$\int g(at)\overline{g(bt)}dt = \bigg|\int_0^1 g(t)dt \bigg|^2.$$
		\end{enumerate}
	\end{thm}
	\begin{rem}\label{rem Delta} 
		In our proofs, the condition which will `pop out' of our calculations is not in fact the one involving $\int g(at)\overline{g(bt)}$. Instead, the more complicated-looking condition $\Delta(g) = 0$ will arise, where
		\begin{equation}\label{eq: def Delta}
			\Delta(g) :=  \sum_{\substack{\l_1,\l_2,\l_3,\l_4 \in \Z \\ \l_1\l_2 = \l_3\l_4 \\ \l_1 \ne \l_3, \l_2 \ne \l_4 \\ \l_1 \l_3,\: \l_2\l_4 > 0}}\frac{\hat{g}(\l_1)\hat{g}(\l_2) \overline{\hat{g}(\l_3)\hat{g}(\l_4)}}{H(\l_1/\l_3)H(\l_2/\l_4)},
		\end{equation}
		where $\hat{g}(r)$ denotes the $r^{th}$ Fourier coefficient of $g$, and $H(\cdot)$ denotes the na\"ive (Weil) height. It is worth pointing out here that $\Delta(g)$ is well defined: the sum defining it is absolutely convergent for any given $g \in \CB\CV[\R/\Z]$, since $\hat{g}(r) \ll_g |r|^{-1}$ for each $r \in \Z_{\ne 0}$ when $g$ is of bounded variation. This bound is sufficient to prove the convergence of $\Delta(g)$. Although $\Delta(g)$ seems rather complicated, we show in \Cref{prop simpl delta} that the condition $\Delta(g) = 0$ is equivalent to the condition appearing in the statement of \Cref{thm:A}:
		$$\Delta(g) = 0\iff \forall a,b \in \N: a \ne b, (a,b) = 1, \text{ one has } \int_0^1 g(at)\overline{g(bt)}dt = \bigg|\int_0^1 g(t)dt\bigg|^2.$$
		Thus, in practice, we will always work with the quantity  $\Delta(g)$. Our goal is then to show that $$S_{g}(N) \xrightarrow[(N \to \infty)]{d} \CC\CN\bigg(0,1-\frac{|\hat{g}(0)|^2}{\|g\|_2^2}\bigg) \iff \Delta(g) = 0.$$
		This is clearly equivalent to \Cref{thm:A}, in light of the present remark.
	\end{rem}
	\Cref{thm:A} will be proven in \Cref{sec:proof if} and \Cref{sec: proof only if}. A few words are in order regarding the proof of the `only if' direction. The core strategy relies on uniformly integrable moments. Suppose, for the sake of contradiction, that $S_g(N)$ converges in distribution to a standard complex Gaussian, but $\Delta(g) \neq 0$ (see \Cref{rem Delta}). To circumvent the analytic difficulties of working with the full weight $g$, we approximate it by a Fourier polynomial $Q_{L_N}$ of slowly growing degree $L_N \to \infty$, yielding a proxy sequence $X_N := S_{Q_{L_N}}(N)$ which must also converge in distribution to the same Gaussian limit. 
	
	By applying martingale inequalities (specifically, the Burkholder--Davis--Gundy inequalities), we will show that the sixth moment of $X_N$ is uniformly bounded. This uniform boundedness is the linchpin of the proof: it guarantees that the sequence of fourth moments $\{|X_N|^4\}_{N \geq 1}$ is uniformly integrable. Because uniform integrability allows us to pass weak convergence to the limit of expectations, the fourth moment of $X_N$ must converge to the fourth moment of the standard complex Gaussian, which is exactly $2$. 
	
	On the other hand, a direct combinatorial expansion of the fourth moment of $X_N$ reveals that it converges asymptotically to $2 + 2\Delta(g)/\|g\|_2^4$. Equating these two limits rigorously forces $\Delta(g) = 0$, yielding the desired contradiction.
	
	Even though $\Delta(g) \ne 0$ unless $g$ has a very specific form, it turns out that an arbitrarily mild roughness condition allows one to recover a Gaussian behavior for arbitrary $g \in \CB\CV[\R/\Z]$, regardless of the value of $\Delta(g)$. We define
	\begin{equation}\label{eq:def S_N rough}
		S_{g}(N;z) := \frac{1}{\sqrt{V_{g}(N;z)}}\sum_{\substack{n \leq N \\ P^-(n) > z}} f(n) g\big(\theta n\big),
	\end{equation}
	where $f$ is sampled from the ensemble of Steinhaus random multiplicative functions, and the variance $V_{g}$ is defined as
	$$V_{g}(N;z) := \E\Big[\Big|\sum_{\substack{n \leq N \\ P^-(n) > z}}f(n)g(\theta n)\Big|^2\Big] = \sum_{\substack{n \leq N \\ P^-(n) > z}}\Big|g\big(\theta n\big)\Big|^2.$$
	With these generalized\footnote{
		By setting $z=1$, we recover the original $S_g(N)$ and $V_g(N)$.
	} definitions in hand, we are ready to state our second theorem. 
	\begin{thm}\label{thm:B}
		Fix $\theta$ to be an irrational number satisfying \eqref{eq:sound-xu badness approx}. Let $g \in \CB\CV_0[\R/\Z]$ be a non-constant function. Let $\xi(N)$ be a function of $N$ which tends to $+\infty$, and satisfies $N-\xi(N) \to +\infty.$ Then $S_{g}(N;\xi(N))$ converges in distribution to a standard complex Gaussian random variable as $N \to \infty$. That is
		$$S_{g}(N;\xi(N)) \xrightarrow[(N \to \infty)]{d} \CC\CN(0,1).$$
	\end{thm}
	\begin{rem}
		\Cref{thm:B} has been stated for $g \in \CB\CV_0[\R/\Z]$. The theorem can be extended to $g \in \CB\CV[\R/\Z]$, but we require an additional assumption on $\xi$. Such an extended theorem is presented in \Cref{cor:B}. The difficulty here is that the analogue of Haper's better-than-squareroot cancellation \eqref{eq:harper better sqrt} over $z$-rough numbers only holds in some ranges of $z$, as was proven by Xu \cite{xu2024better}. See the discussion in \Cref{sec:reduction} for more details.
	\end{rem}
	
	\subsection*{Notation} Throughout the paper we use standard Vinogradov and Oh notation: $\ll,\gg,\sim,\asymp,O(\cdot),o(\cdot)$. We write $e(x) = e^{2 \pi i x}$. The letter $p$ always denotes a prime number. The function $\log_kx$ stands for the $k$-fold iterated logarithm of $x$ (for example, $\log_2 x = \log \log x$). The number $P^+(n)$ denotes the largest prime factor of $n$. The number $P^-(n)$ denotes the smallest prime factor of $n$. The expected value of a random variable $X$ is written as $\E[X]$, and $X_n \to^d X$ means that the sequence of random variables $\{X_n\}$ converges in distribution to $X$. We write $Z\sim \CC\CN(0,\sigma^2)$ to mean that $\E [Z]=0$, $\E [|Z|^2]=\sigma^2$, and $\Re(Z),\Im(Z)$ are independent real Gaussians with variance $\sigma^2/2$. We occasionally use the abbreviation `RMF' (Random Multiplicative Function). 
	Given $a,b$ with $b \ne 0$, the na\"ive (Weil) height is defined as:
	$$H(a/b) = \frac{\max \{|a|,|b|\}}{\gcd(|a|,|b|)}.$$
	Given a function $g \in L^2[\R/\Z]$, we write the total variation of $g$ as $\|g\|_{\CV}$. That is:
	$$\|g\|_{\CV} := \sup_{n \geq 1}\;\bigg(\sup_{\substack{0 \leq x_0 \leq \cdots \leq x_n \leq 1}} \sum_{i = 0}^{n-1} \big|g(x_{i+1})-g(x_i) \big|\bigg)$$
	We then define the set $\CB\CV[\R/\Z]$ to be those one-periodic functions which have bounded total variation, and which are equal to their midpoint at every point of discontinuity:
	$$\CB\CV[\R/\Z] := \bigg\{ g \in L^2[\R/\Z] : \|g\|_{\CV}<\infty, g(t) = \frac{g(t^-)+g(t^+)}{2} \quad \forall t \in [0,1] \bigg\}.$$
	The midpoint normalization ensures that any $g \in \CB\CV[\R/\Z]$ agrees with its Fourier series pointwise on $[0,1]$. We further define $\CB\CV_0[\R/\Z]$ to be those functions in $\CB\CV[\R/\Z]$ which have mean $0$:
	$$\CB\CV_0[\R/\Z] := \bigg\{ g \in \CB\CV[\R/\Z] : \hat{g}(0) = \int_0^1 g(t)dt = 0 \bigg\}.$$
	\section*{Acknowledgments}
	The author was supported by the Natural Sciences and Engineering Research Council of Canada's Doctoral Research Scholarship (NSERC CGRS D - 611593 - 2026) and by the Chaire Courtois en recherche fondamentale II, and by the National Research Council of Canada (RGPIN-2024-05850).
	
	The author thanks Dimitris Koukoulopoulos for his supervision and engagement throughout this project. The author also thanks (in no particular order) Max Xu, Seth Hardy, Tony Haddad, Cihan Sabuncu, Stelios Sachpazis, and Besfort Shala for helpful mathematical discussions. Lastly, the author is grateful to Stephanie Tanasia for her endless support.

	\section{Useful Lemmas}
	
	\subsection{General Lemmas}
	\begin{lem}[{\cite[Lemma 3.4]{sound-xu-CLT}}]\label{lem:parameterization}
		The solutions to $n_1n_2=n_3n_4$ with $n_i \in \N$ may be parameterized as
		$$n_1 = ga, \; n_2 = hb, \; n_3 = gb, \; n_4 = ha,$$
		where $a,b,g,h \in \N$, and $(a,b) = 1$. This parameterization is a $1$-$1$ correspondence. Diagonal solutions, meaning solutions for which the multisets ${n_1,n_2}$ and ${n_3,n_4}$ are equal, correspond to solutions with $g = h$ or $a = b$ (in which case $a = b = 1$).
	\end{lem}
	
\begin{lem}\label{prop simpl delta}
	Let $g \in \CB\CV[\R/\Z]$ be given. With $\Delta$ defined as in \eqref{eq: def Delta}, we have
	$$\Delta(g) = \sum_{\substack{a,b \in \N \\ a \ne b \\ (a,b) = 1}}\frac{1}{H(a/b)^2} \bigg|\int_0^1 g(at)\overline{g(bt)}dt - \Big|\int_0^1 g(t)dt\Big|^2\bigg|^2.$$
	In particular,
	\begin{align*}
		\Delta(g) = 0 & \iff \forall a,b \in \N: a \ne b, (a,b) = 1, \text{ one has } \int_0^1 g(at)\overline{g(bt)}dt =\Big|\int_0^1 g(t)dt\Big|^2.  
	\end{align*}
\end{lem}

\begin{proof}
	First of all, it is a standard fact from Fourier analysis that any $g \in \CB\CV[\R/\Z]$ satisfies the bound $|\hat{g}(k)| \ll_g 1/|k|$ for all $k \in \Z_{\ne 0}$. In particular, this implies that 
	$$|\Delta(g)| \ll_g 1+ \sum_{\substack{\l_1,\l_2, \l_3,\l_4 \in \Z \\ \l_1\l_2 = \l_3\l_4}} \frac{1}{|\l_1\l_2\l_3\l_4|} \ll \sum_{\substack{\l > 0 }} \frac{\tau(\l)^2}{\l^2} < \infty.$$
	As such, the sums defining $\Delta(g)$ converge absolutely for any given $g$. This justifies manipulating and rearranging the sums as we will now do. Let $\xi_1 = \text{sign}(\l_1) = \text{sign}(\l_3),$ and let $\xi_2 = \text{sign}(\l_2) = \text{sign}(\l_4).$ By the definition of $\Delta(g)$ given in \eqref{eq: def Delta}, we can write
	$$\Delta(g) = \sum_{\xi_1,\xi_2 \in \{\pm 1\}}\sum_{\substack{\l_1,\l_2,\l_3,\l_4 \in \Z_{>0} \\ \l_1\l_2 = \l_3\l_4 \\ \l_1 \ne \l_3, \l_2 \ne \l_4}}\frac{\hat{g}(\xi_1\l_1)\hat{g}(\xi_2\l_2) \overline{\hat{g}(\xi_1\l_3)\hat{g}(\xi_2\l_4)}}{H(\l_1/\l_3)H(\l_2/\l_4)}. $$
	
	By \Cref{lem:parameterization}, we can parameterize the $\l_i$ as
	$$\l_1 = ra, \; \l_2 = sb, \; \l_3 = rb, \; \l_4 = sa,$$
	where $a,b,r,s \in \Z_{>0}$, $(a,b) = 1,$ and $a \ne b$ since $\l_1 \ne \l_3$ and $\l_2 \ne \l_4$. Using this parameterization, we have
	$$\Delta(g) = \sum_{\xi_1,\xi_2 \in \{\pm 1\}}\sum_{\substack{a,b,r,s \in \Z_{>0} \\ a \ne b \\ (a,b) = 1}}\frac{\hat{g}(\xi_1ra)\hat{g}(\xi_2sb) \overline{\hat{g}(\xi_1rb)\hat{g}(\xi_2sa)}}{H(a/b)^2}, $$
	since $H(a/b)=H(b/a)$. Taking the $r,s$ sums to the inside gives
	\begin{align*}
		\Delta(g) &= \sum_{\xi_1,\xi_2 \in \{\pm 1\}}\sum_{\substack{a,b \in \Z_{>0} \\ a \ne b \\ (a,b) = 1}}\frac{1}{H(a/b)^2} \sum_{r,s \in \Z_{>0}}\hat{g}(\xi_1ra)\hat{g}(\xi_2sb) \overline{\hat{g}(\xi_1rb)\hat{g}(\xi_2sa)} \\
		&=\sum_{\substack{a,b \in \Z_{>0} \\ a \ne b \\ (a,b) = 1}}\frac{1}{H(a/b)^2} \bigg(\sum_{r \in \Z_{>0}} \sum_{\xi_1 \in \{\pm 1\}}\hat{g}(\xi_1ra)\overline{\hat{g}(\xi_1rb)} \bigg)\overline{\bigg(\sum_{s \in \Z_{>0}} \sum_{\xi_2 \in \{\pm 1\}}\hat{g}(\xi_2 sa)\overline{\hat{g}(\xi_2 sb)} \bigg)}\\
		&=\sum_{\substack{a,b \in \Z_{>0} \\ a \ne b \\ (a,b) = 1}}\frac{1}{H(a/b)^2} \bigg|\sum_{r \in \Z_{>0}} \sum_{\xi_1 \in \{\pm 1\}}\hat{g}(\xi_1ra)\overline{\hat{g}(\xi_1rb)} \bigg|^2.
	\end{align*}
	
	Since $r$ ranges over all strictly positive integers, then the double sum of $\xi_1$ and $r$ can be rewritten as a single sum of $r$, now ranging over all nonzero integers. This proves that 
	$$\Delta(g) = \sum_{\substack{a,b \in \N \\ a \ne b \\ (a,b) = 1}}\frac{1}{H(a/b)^2} \bigg|\sum_{r \in \Z_{\ne 0}} \hat{g}(ra) \overline{\hat{g}(rb)}\bigg|^2.$$
	
	By adding and subtracting the $r=0$ term, we can write the inner sum over all integers $\Z$:
	$$\sum_{r \in \Z_{\ne 0}} \hat{g}(ra) \overline{\hat{g}(rb)} = \sum_{r \in \Z} \hat{g}(ra) \overline{\hat{g}(rb)} - |\hat{g}(0)|^2.$$
	 Let $F_N$ denote the $N$-th Fej\'er kernel and define the sequence of trigonometric polynomials $g_N = g * F_N$. The Fourier coefficients of $g_N$ are strictly supported on $|n| \le N$, given explicitly by $\widehat{g_N}(n) = (1 - |n|/N)\hat{g}(n)$ for $|n| \le N$ and zero otherwise. 
	
	For these truncated polynomials, the corresponding double sum is finite, so we may safely interchange summation and integration without absolute convergence concerns:
	\begin{align*}
		\sum_{r \in \Z} \widehat{g_N}(ra) \overline{\widehat{g_N}(rb)} &= \sum_{n_1,n_2 \in \Z} \bm{1}_{bn_1 = an_2} \widehat{g_N}(n_1) \overline{\widehat{g_N}(n_2)} \\
		&= \int_0^1 \bigg(\sum_{n_1 \in \Z} \widehat{g_N}(n_1)e(n_1bt)\bigg)\bigg(\sum_{n_2 \in \Z} \overline{\widehat{g_N}(n_2)}e(-n_2at)\bigg)dt \\
		&= \int_0^1 g_N(bt)\overline{g_N(at)}dt.
	\end{align*}
	Above, we have used the fact that $\bm{1}_{X = Y} = \int_0^1 e(t(X-Y))dt$. We now take the limit as $N \to \infty$ on both sides. On the left-hand side, since $g \in \CB\CV[\R/\Z]$, we have the bound $|\hat{g}(k)| \ll_g 1/|k|$ for $k \ne 0$. This gives the dominating bound:
	$$|\widehat{g_N}(ra) \overline{\widehat{g_N}(rb)}| \le |\hat{g}(ra)\hat{g}(rb)| \ll_g \frac{1}{ab r^2},$$
	which is absolutely summable over $r \in \Z_{\ne 0}$. By the Dominated Convergence Theorem for counting measures, the left-hand side converges exactly to our target sum:
	$$\lim_{N \to \infty} \sum_{r \in \Z} \widehat{g_N}(ra) \overline{\widehat{g_N}(rb)} = \sum_{r \in \Z} \hat{g}(ra) \overline{\hat{g}(rb)}.$$
	
	On the right-hand side, because $g$ is of bounded variation, it is bounded and thus belongs to $L^2[\R/\Z]$. A standard property of the Fej\'er kernel is that $g_N \to g$ in $L^2[\R/\Z]$. By adding and subtracting cross terms and applying the Cauchy-Schwarz inequality, we can bound the difference of the integrals:
	\begin{align*}
		\bigg| \int_0^1 g_N(bt)\overline{g_N(at)}dt - \int_0^1 g(bt)\overline{g(at)}dt \bigg| &\le \int_0^1 |g_N(bt) - g(bt)||\overline{g_N(at)}| dt \\
		& \quad + \int_0^1 |g(bt)||\overline{g_N(at)} - \overline{g(at)}| dt \\
		&\le \|g_N(b \cdot) - g(b \cdot)\|_{L^2} \|g_N(a \cdot)\|_{L^2}\\ & \quad + \|g(b \cdot)\|_{L^2} \|g_N(a \cdot) - g(a \cdot)\|_{L^2}.
	\end{align*}
	Above, $g(a \cdot)$ indicates the function $t \to g(at)$ and likewise for $g(b \cdot)$. By the 1-periodicity of these functions, the $L^2$ norm over $[0,1]$ is invariant under non-zero integer dilations, so $\|f(c \cdot)\|_{L^2} = \|f\|_{L^2}$. The bound simplifies to:
	$$\|g_N - g\|_{L^2} \|g_N\|_{L^2} + \|g\|_{L^2} \|g_N - g\|_{L^2}.$$
	Since $\|g_N - g\|_{L^2} \to 0$ and $\|g_N\|_{L^2}$ is uniformly bounded, the right-hand side converges to $\int_0^1 g(bt)\overline{g(at)}dt$. Equating the two limits establishes that:
	$$\sum_{r \in \Z} \hat{g}(ra) \overline{\hat{g}(rb)} = \int_0^1 g(bt)\overline{g(at)}dt.$$
	
	Substituting this equality and the shifted $|\hat{g}(0)|^2$ term back into our expression for $\Delta(g)$, we obtain:
	\begin{align*}
		\Delta(g) = \sum_{\substack{a,b \in \N \\ a \ne b \\ (a,b) = 1}}\frac{1}{H(a/b)^2} \bigg|\int_0^1 g(bt)\overline{g(at)}dt - |\hat{g}(0)|^2\bigg|^2.
	\end{align*}
	
	Note that taking the complex conjugate inside the absolute value yields the equivalent expression with $\int_0^1 g(at)\overline{g(bt)}dt$. This completes the first claim of the proof. The second claim follows immediately from the first, since a sum of non-negative terms is zero if and only if all terms are zero.
\end{proof}
	
	\subsection{Diophantine Spacing}
	
	\begin{lem}[Diophantine Spacing and Divisor Weights]\label{lem: badness spacing}
		Let $\theta \in \R \setminus \Q$ satisfy the Diophantine condition \eqref{eq:sound-xu badness approx}, i.e. 
		$$	\|q\theta\|:= \min_{n \in \Z}|q \theta -n |\geq C\exp(-|q|^{1/127}) \quad \forall q \in \Z \setminus \{0\}.$$
		Let $X \ge 2$, $Q \ge 1$, and $\eta \in (0, 1/2)$, such that $\eta Q \in (0,1/2)$. Let $s, y$ be integers with $s \ne 0$. Define the set
		$$ \mathcal{B} = \Big\{ b \le X : \exists q \in \Z_{\ne 0}, |q| \le Q \text{ s.t. } \|\theta q(b s - y)\| \le \eta \Big\}. $$
		Then any distinct elements $b, b' \in \mathcal{B}$ satisfy $|b - b'| \ge \delta$, where
		$$ \delta = \max\bigg\{1,\frac{1}{|s| Q^2} \bigg( \log \frac{C_1}{2 Q \eta} \bigg)^{127} \bigg \}, $$
		for some large absolute constant $C_1$. Consequently, we have the cardinality bound $|\mathcal{B}| \ll 1 + X/\delta$. Furthermore, for any integer $k \ge 1$, we have
		$$ \sum_{b \in \mathcal{B}} \tau_k(b) \ll_k X \delta^{-1/2} (\log X)^{\frac{k^2-1}{2}} + X^{1/2} (\log X)^{\frac{k^2-1}{2}}. $$
	\end{lem}
	\begin{proof}
		Suppose $b, b' \in \mathcal{B}$ are distinct. Then there exist non-zero $q_1, q_2 \in [-Q, Q]$ such that $\|\theta q_1(bs - y)\| \le \eta$ and $\|\theta q_2(b's - y)\| \le \eta$. By the triangle inequality, $\|\theta q_1 q_2 s (b - b')\| \le 2Q\eta$. Applying condition \eqref{eq:sound-xu badness approx}, we have
		$$ C_1 \exp\big(-|q_1 q_2 s (b - b')|^{1/127}\big) \le 2Q\eta. $$
		Rearranging this inequality and using $|q_1 q_2| \le Q^2$ yields the claimed lower bound $\delta$ on the spacing $|b - b'|$. The cardinality bound $|\mathcal{B}| \ll 1 + X/\delta$ follows immediately. 
		
		To bound the divisor sum, we apply the Cauchy-Schwarz inequality over the interval $[1, X]$:
		\begin{align*}
			\sum_{b \in \mathcal{B}} \tau_k(b) &\le \bigg( \sum_{b \in \mathcal{B}} 1 \bigg)^{1/2} \bigg( \sum_{b \le X} \tau_k(b)^2 \bigg)^{1/2} \\
			&\ll_k \big( 1 + X/\delta \big)^{1/2} X^{1/2} (\log X)^{\frac{k^2-1}{2}} \\
			&\le \big( X^{1/2} + X \delta^{-1/2} \big) (\log X)^{\frac{k^2-1}{2}},
		\end{align*}
		having used the standard estimate $\sum_{n \le X} \tau_k(n)^2 \ll_k X (\log X)^{k^2-1}$ and the inequality $\sqrt{A+B} \le \sqrt{A} + \sqrt{B}$.
	\end{proof}
	
	\begin{cor}\label{cor: log weights}
		Let $\mathcal{B}$ and $\delta \ge 1$ be as in Lemma \ref{lem: badness spacing}. For any $Y \ge 1$, we have
		$$ \sum_{\substack{b \in \mathcal{B} \\ b \ge Y}} \frac{\tau_k(b)}{b} \ll_k \delta^{-1/2} (\log X)^{\frac{k^2+1}{2}} + Y^{-1/2} (\log X)^{\frac{k^2-1}{2}}. $$
	\end{cor}
	\begin{proof}
		We partition the sum into intervals $b \in (e^{j-1},e^j]$ for $\log Y \le j \le \log X + 1$. Applying the divisor sum bound from Lemma \ref{lem: badness spacing} to each interval gives:
		$$ \sum_{\substack{b \in \mathcal{B} \cap (e^{j-1},e^j]}} \tau_k(b) \ll_k \big( e^j \delta^{-1/2} + e^{j/2} \big) j^{\frac{k^2-1}{2}}. $$
		Dividing by $e^j$ and summing over $j$ yields
		\begin{align*}
			\sum_{\substack{b \in \mathcal{B} \\ b \ge Y}} \frac{\tau_k(b)}{b} &\ll_k \sum_{\log Y \le j \le \log X + 1} \Big( \delta^{-1/2} + e^{-j/2} \Big) j^{\frac{k^2-1}{2}} \\
			&\ll_k \delta^{-1/2} (\log X)^{\frac{k^2+1}{2}} + Y^{-1/2} (\log Y)^{\frac{k^2-1}{2}}.
		\end{align*}
	\end{proof}

	\subsection{Tools From Sieve Theory}
	
	\begin{lem}\label{lem: smooth divisor sums}
		Let $A > 0$ be given. For $N$ sufficiently large in terms of $A$, one has $$\sum_{p \leq N} \bigg(\sum_{\substack{m \leq N/p \\ P^+(m) \leq p}} \tau(m)\bigg)^2 \ll_A \frac{N^2}{(\log N)^A}.$$
	\end{lem}
	\begin{proof}[Proof of \Cref{lem: smooth divisor sums}]
		Clearly, if the result holds for $A > 3$ then it also holds for $A \in (0,3]$. Thus, suppose without loss of generality that $A > 3$. We split $p$ into two ranges:
		\begin{equation}\label{eq:smooth divisor sums 1}\sum_{p \leq N} \bigg(\sum_{\substack{m \leq N/p \\ P^+(m) \leq p}} \tau(m)\bigg)^2 = \sum_{p \leq (\log N)^{A+2}} \bigg(\sum_{\substack{m \leq N/p \\ P^+(m) \leq p}} \tau(m)\bigg)^2+\sum_{(\log N)^{A+2} < p \leq N} \bigg(\sum_{\substack{m \leq N/p \\ P^+(m) \leq p}} \tau(m)\bigg)^2
		\end{equation}
		When $p > (\log N)^{A+2}$, we ignore the smoothness condition and write
		\begin{align*}
			\sum_{p > (\log N)^{A+2}} \bigg(\sum_{\substack{m \leq N/p \\ P^+(m) \leq p}} \tau(m)\bigg)^2 &\ll \sum_{p > (\log N)^{A+2}} \bigg(\sum_{\substack{m \leq N/p}} \tau(m)\bigg)^2\\
			&\ll N^2 (\log N)^2 \sum_{p > (\log N)^{A+2}} \frac{1}{p^2} .
		\end{align*} 
		Since $\sum_{p > T}1/p^2 \ll 1/T,$ we conclude that
		\begin{equation}\label{eq:smooth divisor sums 2}
			\sum_{p >(\log N)^{A+2}} \bigg(\sum_{\substack{m \leq N/p \\ P^+(m) \leq p}} \tau(m)\bigg)^2 \ll \frac{N^2}{(\log N)^A}.
		\end{equation}
		When $p \le (\log N)^{A+2}$, we apply Rankin's trick. For any $\sigma > 0$, we have
		$$\sum_{\substack{m \leq N/p \\ P^+(m) \leq p}} \tau(m) \leq \frac{N^\sigma}{p^\sigma} \sum_{\substack{m \geq 1 \\ P^+(m) \leq p}} \frac{\tau(m)}{m^\sigma} = \frac{N^\sigma}{p^\sigma} \prod_{q \leq p}\bigg(1-\frac{1}{q^\sigma}\bigg)^{-2}. $$
		By choosing $\sigma = 1-1/(2A)$, we obtain
		$$\sum_{\substack{m \leq N/p \\ P^+(m) \leq p}} \tau(m) \ll \frac{N^{1-1/(2A)}}{p^{1-1/(2A)}} \exp(4Ap^{\frac{1}{2A}}).$$
		As such, we have
		\begin{align*}
			\sum_{p \le (\log N)^{A+2}}\bigg(\sum_{\substack{m \leq N/p \\ P^+(m) \leq p}} \tau(m)\bigg)^2 &\ll N^{2-\frac{1}{A}}\sum_{p \le (\log N)^{A+2}}\frac{\exp(8Ap^{\frac{1}{2A}})}{p^{2-1/A}}\\ & \ll N^{2-\frac{1}{A}}\exp(8A(\log N)^{\frac{A+2}{2A}})\sum_{p  \le (\log N)^{A+2}}\frac{1}{p^{2-1/A}} 
		\end{align*}
		Since we assumed that $A > 3$, the $p$ sum above is $\ll 1$, and $(A+2)/(2A) \leq 5/6$, so we conclude that
		\begin{align}\label{eq:smooth divisor sums 3}
			\sum_{p \leq (\log N)^{A+2}}\bigg(\sum_{\substack{m \leq N/p \\ P^+(m) \leq p}} \tau(m)\bigg)^2 &\ll N^{2-\frac{1}{A}}\exp(8A(\log N)^{5/6}),
		\end{align}
		which is much smaller than the required $N^2/(\log N)^A$. We complete the proof of the lemma by combining \eqref{eq:smooth divisor sums 2} and \eqref{eq:smooth divisor sums 3} back into \eqref{eq:smooth divisor sums 1}.
	\end{proof}
	The next lemma provides a method for sieving the kind of objects we will encounter in this paper. A proof can be found in the paper of Sachpazis \cite[Lemma 3.8]{stelios}. We note that we have specialized the lemma to our specific case, and the version in the cited reference is more general.
	\begin{lem}\label{lem: stelios}
		Let $z \geq 2$ be given, and set $P_z =\prod_{p \leq z}p$. Suppose $u \geq 5$. Then there exist two arithmetic functions $\lambda^+,\lambda^-$ such that:
		\begin{enumerate}
			\item $\lambda^\pm(1) = 1, |\lambda^\pm| \leq 1$,
			\item $\supp (\lambda^\pm) \subseteq \{d| P_z : d \leq z^u\}$,
			\item $(1*\lambda^-)(n) \leq \bm{1}_{P^-(n) > z} \leq (1*\lambda^+)(n)$ for all $n \in \N$,
			\item $\sum_{d | P_z} \frac{\lambda^\pm(d)}{d} = \sum_{d | P_z} \frac{\mu(d)}{d} + O\big(u^{-u/2}\big).$ 
		\end{enumerate}
	\end{lem}

	\subsection{Exponential Sums}
	For the proofs of all our theorems, we require estimates on exponential sums over different sifted sets. The results stated in this subsection accomplish this.\footnote{\Cref{thm: exp sum primes} is stated in \cite{chen1985estimation_chinese} with $y=2$. If $y<q$, then if suffices to use the trivial bound $\sum_{p\le y} e(\theta p)\ll q$, which is smaller than the term $x\sqrt{q/x}$ on the right hand side. On the other hand, if $y\ge q$, then we use twice the original Chen result. We note also that the proofs of our main theorems still work if \Cref{thm: exp sum primes}~is replaced with the classical estimate of Vinogradov on exponential sums over primes, but doing so would result in a quantitatively weaker version of \Cref{thm: bounded moments}, and hence the diophantine condition \eqref{eq:sound-xu badness approx} would have to be strengthened correspondingly.}
	\begin{lem}[Chen, 1985 \cite{chen1985estimation_chinese}]\label{thm: exp sum primes}
		Suppose $\theta = \frac{a}{q} + \frac{\beta}{q^2},$ where $|\beta| \leq 1$. For all sufficiently large $x$, one then has:
		$$\bigg| \sum_{y < p \leq x}e(\theta p) \bigg| \ll x (\log x)^{3/4 } (\log \log x) \bigg( \sqrt{\frac{1}{q} + \frac{q \log q}{x}} + \sqrt{\log q} \exp \big(-\frac{1}{2} \sqrt{\log x}\big) \bigg) + q.$$
	\end{lem}
	
	\begin{lem}[{Koukoulopoulos, \cite[Theorem 23.5]{koukoulopoulos} }]\label{lem: dimitris exp sums}
		Let $f$ be supported on $[1,y]$, $v \geq 0,$  $x \geq 2,$ $\alpha \in \R$ and $a/q$ be a reduced fraction with $|\alpha -a/q| \leq 1/q^2$. Then 
		$$\sum_{n \leq x} (f * \log^v)(n) e(n \alpha) \ll \Big(y + \frac{x}{q} + q\Big)(\log x)^{v+1} ||f||_\infty.$$
	\end{lem}
	
	The next theorem gives bounds for exponential sums over rough numbers. Although this result is proved by standard methods, it seems to be new in the sense that an estimate for exponential sums over rough numbers does not appear in the literature, as far as the author is aware. We also remark that the bound can likely be strengthened substantially in some regimes (especially large $z$), and even extended to multiplicative coefficient weights. We do not pursue this here.
	\begin{thm} \label{exp sum rough}
		Let $x \geq 2$, and suppose $\theta = \frac{a}{q} + \frac{\beta}{q^2},$ where $|\beta| \leq 1$ and $1 \leq q \leq x$. We have uniformly for $z \in [2,x/2]$ that 
		$$\sum_{\substack{n \leq x \\ P^-(n) > z}} e(\theta n) \ll x \exp\bigg(-\frac{1}{3}\sqrt{\log x}\bigg) +\bigg(\frac{x}{\sqrt{q}} + q\bigg)(\log x)^{5/2}.$$
	\end{thm}
	\begin{proof}
		We first deal with the case when $z \leq \exp(\sqrt{\log x})$. Write $P_z = \prod_{p \leq z}p.$ \Cref{lem: stelios} implies that for any $u \geq 5$ there exist weights $\lambda^\pm(d)$ such that $|\lambda^\pm(d)| \leq 1$,
		$$
		\supp(\lambda^\pm) \subseteq \{d \mid P_z : d \leq z^u\},
		$$
		$$
		(1*\lambda^-)(n) \leq \bm{1}_{(n,P_z)=1} \leq (1*\lambda^+)(n),
		$$
		and
		\begin{equation}\label{eq: stelios 1}
			\sum_{d \mid P_z} \frac{\lambda^{\pm}(d)}{d}
			=
			\sum_{d \mid P_z} \frac{\mu(d)}{d}
			+
			O\big( u^{-u/2} \big).
		\end{equation}
		We write
		\begin{equation}\label{eq: stelios 2}
			\sum_{\substack{n \leq x \\ P^-(n)>z}} e(\theta n)
			=
			\sum_{n \leq x} e(\theta n)
			\big(\bm{1}_{(n,P_z)=1}-(1*\lambda^-)(n)\big)
			+
			\sum_{n \leq x} e(\theta n)(1*\lambda^-)(n).
		\end{equation}
		We bound each sum separately. For the first sum, since
		$$
		(1*\lambda^-)(n) \leq \bm{1}_{(n,P_z)=1} \leq (1*\lambda^+)(n),
		$$
		we have
		$$
		\left|\bm{1}_{(n,P_z)=1}-(1*\lambda^-)(n)\right|
		\leq
		(1*\lambda^+)(n)-(1*\lambda^-)(n).
		$$
		Therefore
		$$
		\left|
		\sum_{n \leq x} e(\theta n)
		\big(\bm{1}_{(n,P_z)=1}-(1*\lambda^-)(n)\big)
		\right|
		\leq
		\sum_{n \leq x}
		\big((1*\lambda^+)(n)-(1*\lambda^-)(n)\big).
		$$
		Opening the convolutions and changing the order of summation gives
		\begin{align*}
			\sum_{n \leq x}
			\big((1*\lambda^+)(n)-(1*\lambda^-)(n)\big)
			&=
			\sum_{d \leq x}
			(\lambda^+(d)-\lambda^-(d))
			\left\lfloor \frac{x}{d} \right\rfloor  \\
			&=
			x\sum_{d \leq x}\frac{\lambda^+(d)-\lambda^-(d)}{d}
			+
			O\left(\sum_{d \leq x}|\lambda^+(d)-\lambda^-(d)|\right).
		\end{align*}
		Since $\lambda^\pm$ are supported on $d\leq z^u$ and satisfy $|\lambda^\pm(d)|\leq 1$, the error term is $O(z^u)$. Moreover, if we assume that $z^u \leq x$, then by \eqref{eq: stelios 1}, we have
		$$
		\sum_{d \leq x}\frac{\lambda^+(d)-\lambda^-(d)}{d}
		=
		O(u^{-u/2}).
		$$
		Hence
		\begin{equation}\label{eq: rough first error}
			\left|
			\sum_{n \leq x} e(\theta n)
			\big(\bm{1}_{(n,P_z)=1}-(1*\lambda^-)(n)\big)
			\right|
			\ll
			xu^{-u/2}+z^u .
		\end{equation}
		To bound the second term in \eqref{eq: stelios 2}, we apply \Cref{lem: dimitris exp sums} with $y = z^u, f = \lambda^-, v = 0, \alpha = \theta.$ This implies
		\begin{align*}
			\left|
			\sum_{n \leq x} e(\theta n)(1*\lambda^-)(n)
			\right|
			&\ll
			\big(z^u+xq^{-1}+q\big)\log x .
		\end{align*}
		Combining this with \eqref{eq: rough first error}, we obtain
		$$
		\sum_{\substack{n \leq x \\ P^-(n)>z}} e(\theta n)
		\ll
		xu^{-u/2}
		+
		\big(z^u+xq^{-1}+q\big)\log x .
		$$
		Choose $u=\frac{\log x}{2 \log z},$ so that $z^u = \sqrt{x}$, and $x u^{-u/2} \ll x \exp \big(- \frac{1}{3} \sqrt{\log x}\big)$, since we assumed that $z \leq \exp(\sqrt{\log x})$. These bounds imply that 
		$$
		\sum_{\substack{n \leq x \\ P^-(n)>z}} e(\theta n)
		\ll
		x\exp\bigg(- \frac{1}{3} \sqrt{\log x}\bigg)		+
		\big(xq^{-1}+q\big)\log x, \quad z \leq \exp(\sqrt{\log x}).
		$$
		
		It remains to consider the range $z \ge \exp(\sqrt{\log x}).$ Since \Cref{thm: exp sum primes} already covers the case of prime sums, it suffices to consider $z \in [\exp(\sqrt{\log x}), \sqrt{x}].$ Here, we make use of Ramar\'e's identity. Define
		$$\beta_n(x) := \frac{1}{1+\omega(n,\sqrt{x})},$$
		where $\omega(n,t)$ denotes the number of distinct $p|n$ such that $p \leq t$. If $n \in (1,x]$ is squarefree, one has
		$$\sum_{\substack{p \leq \sqrt{x} \\ p | n}} \beta_{n/p}(x) = \begin{cases}
			1 \quad \text{if } n \text{ has a prime factor }\leq \sqrt{x}\\
			0 \quad \text{if } n \text{ is prime with } \sqrt{x} < n \leq x
		\end{cases}.$$
		Thus, for any sequence $\{a_n\}$ of complex numbers, one has 
		\begin{equation}\label{eq: ramare summed}
			\sum_{\sqrt{x} < p \leq x}a_p = \sum_{\substack{1< n \leq x \\ \mu^2(n) = 1}} a_n - \underset{ \substack{p \leq \sqrt{x},\;mp\leq x \\ (m,p) = 1 \\ \mu^2(m) = 1}}{\sum \sum}\beta_m(x)a_{mp}.
		\end{equation}
		If we set $a_n := \bm{1}_{(n,P_z)=1}e(\theta  n),$ then \eqref{eq: ramare summed} becomes
		\begin{equation}\label{eq: 1}
			\sum_{\sqrt{x} < p \leq x}e(\theta p) = \sum_{\substack{1< n \leq x \\ (n,P_z)=1 \\ \mu^2(n) = 1}} e(\theta n) - \underset{ \substack{z< p \leq \sqrt{x},\;mp\leq x \\ p \nmid m,\; (m,P_z) = 1 \\ \mu^2(m) = 1}}{\sum \sum}\beta_m(x)e(\theta mp).
		\end{equation}
		
		In the first sum on the right hand side of the above equation, one can drop the condition $\mu^2(n) = 1$ at the cost of a certain error. Indeed,
		\begin{align*}
			\sum_{\substack{1< n \leq x \\ (n,P_z)=1 \\ \mu^2(n) = 1}} e(\theta n) &= \sum_{\substack{1< n \leq x \\ (n,P_z)=1}} e(\theta n) + O\bigg(\sum_{p > z} \sum_{\substack{m \leq x/p^2 \\ P^-(m) > z}}1\bigg)\\
			&= \sum_{\substack{1< n \leq x \\ (n,P_z)=1}} e(\theta n) + O\bigg(\frac{x}{z \log z}\bigg).
		\end{align*}
		Substituting these into \eqref{eq: 1} and rearranging the equality gives
		\begin{equation}\label{eq: 2}
			\bigg|\sum_{\substack{n \leq x \\ P^-(n) > z}} e(\theta n) \bigg| \le \bigg|\sum_{\sqrt{x} < p \leq x}e(\theta p)\bigg| + \bigg| \underset{ \substack{z< p \leq \sqrt{x},\;mp\leq x \\ (m,P_z) = 1 \\ \mu^2(m) = 1}}{\sum \sum}\beta_m(x)e(\theta mp) \bigg| + O\bigg(\frac{x}{z \log z}\bigg).
		\end{equation}
		We bound each sum on the right hand side individually. By \Cref{thm: exp sum primes}, we immediately have an appropriate bound on the sum over $p$. To bound the second (type II) sum, we decompose the $m$ and $p$ sums into dyadic ranges where $m \sim M$ and $p \sim P$, with $MP \leq x$. For ease of notation, write $\alpha_m = \mu^2(m) \bm{1}_{P^-(m) > z} \beta_m(x).$ For a given dyadic box, the sum can be written as 
		
		$$\CB(M,P) = \sum_{\substack{m \sim M }} \alpha_m \sum_{p \sim P}e(\theta m p).$$
		By Cauchy-Schwarz, we have
		$$|\CB(M,P)|^2 \leq \bigg(\sum_{\substack{m \sim M}} |\alpha_m|^2\bigg)\times\bigg( \sum_{m \sim M}\bigg|\sum_{p \sim P}e(\theta m p)\bigg|^2\bigg).$$
		Since $\beta_m(x) \ll 1$ uniformly in $m,x$, the first factor above can be bounded as
		$$\sum_{\substack{m \sim M}} |\alpha_m|^2 \ll \sum_{\substack{
				m \sim M \\ P^-(m) > z}} 1 \ll \frac{M}{\log z}.$$
		To bound the second factor, we use the fact that $|z|^2 = z \overline{z}$ to write
		
		\begin{align*}
			\sum_{m \sim M}\bigg|\sum_{p \sim P}e(\theta m p)\bigg|^2 &=  \sum_{m \sim M}\sum_{p_1,p_2 \sim P}e(\theta m(p_1-p_2))\\
			&= \sum_{m \sim M}\sum_{p \sim P}1 + \sum_{\substack{p_1,p_2 \sim P \\ p_1 \ne p_2}} \min(M,\|\theta(p_1-p_2)\|^{-1})\\
			&\ll \frac{MP}{\log P} + \sum_{k \leq P}\gamma_{k} \min\{M,\|\theta k\|^{-1}\},
		\end{align*}
		where $\gamma_{k} := \#\{p_1,p_2 \leq P: p_1-p_2=k\}.$ We have $\gamma_k \ll \pi(P)$, and so the prime number Theorem implies that $\gamma_k \ll \frac{P}{\log{ P}}.$ It is known (see \cite[p. 346]{iwaniec}) that
		$$\sum_{k \leq P}\min\{M, \|\theta k\|^{-1}\} \ll \left(M+P+\frac{MP}{q}+q\right)\log q.$$
		Using these bounds, we find
		\begin{align*}
			\sum_{m \sim M}\bigg|\sum_{p \sim P}e(\theta m p)\bigg|^2 &\ll \frac{MP }{\log P} + \left(MP+P^2+\frac{MP^2}{q}+qP\right)\frac{\log q }{\log P}
		\end{align*}
		Therefore if we combine our estimates and use the fact that $MP \leq x$, we get
		\begin{eqnarray*}|\CB(M,P)|^2 &\ll& \bigg(\frac{M}{\log z}\bigg)\times\left[\frac{MP }{\log P} +  \left(MP+P^2+\frac{MP^2}{q}+qP\right)\frac{\log q }{\log P}\right]  \\
			&\ll& \frac{x}{\log z} \times \left[\frac{M}{\log z}+\left(M+P+\frac{x}{q}+q\right)\frac{\log q}{\log P}\right]  
		\end{eqnarray*}
		Note that there are at most $(\log x)^2$ dyadic intervals to sum over, and so by using a supremum bound, we find that
		\begin{align*}
			\bigg| \underset{ \substack{z< p \leq \sqrt{x},\;mp\leq x \\ p \nmid m,\; (m,P_z) = 1 \\ \mu^2(m) = 1}}{\sum \sum}\beta_m(x)e(\theta mp) \bigg| \ll (\log x)^2 \sup_{\substack{M \ll x/z \\ z \ll P \ll \sqrt{x}}}|\CB(M,P)|  \ll x\frac{(\log x)^{5/2}}{(\log z)^{3/2}}\left(\frac{1}{z}+\frac{1}{q}+\frac{q}{x}\right)^{1/2}
		\end{align*}
		Combining the estimates for our Type I and Type II sums back into \eqref{eq: 2}, we end up with the estimate
		$$\bigg|\sum_{\substack{n \leq x \\ P^-(n) > z}} e(\theta n) \bigg| \ll x(\log x)^{5/2} \bigg( \sqrt{\frac{1}{q} + \frac{q}{x}} + x^{-1/5}+z^{-1/2}\bigg).$$
		We note that the term $O(x/z\log z)$ is omitted since it is smaller than the last term of the above expression. Since we are in the range $z \in (\exp(\sqrt{\log x}), \sqrt{x}],$ we conclude that 
		$$\bigg|\sum_{\substack{n \leq x \\ P^-(n) > z}} e(\theta n) \bigg| \ll x\exp \bigg( - \frac{1}{3} \sqrt{\log x} \bigg) + x(\log x)^{5/2} \bigg( \sqrt{\frac{1}{q} + \frac{q}{x}} \bigg), \quad z \in (\exp(\sqrt{\log x}), \sqrt{x}].$$
		Combining this with the estimate in the case when $z \leq \exp(\sqrt{\log x})$ completes the proof of the theorem.
	\end{proof}
	
	The next theorem is a slight modification of the ideas presented in \cite{granville breteche exp sums}. The main takeaway is that one can save an arbitrary power of $\log x$ for exponential sums over smooth and/or rough numbers, so long as $q$ is not too small.
	\begin{thm} \label{thm: exp sum smooth}
		Let $A > 1$ be given. Uniformly for $3/2 < z \leq y \leq x$, all real $\theta$ and all $a,q$ with $(a,q)=1$, $q \leq x$, and $|\theta -a/q| \leq 1/q^2$, we have
		$$\bigg|\sum_{\substack{n \leq x \\ P^+(n) \leq y \\ P^-(n) > z}} e(\theta n)\bigg| \ll \frac{x}{(\log x)^{A}} +x (\log x)^{2+A} \bigg(\frac{1}{\sqrt{q}}+\sqrt{\frac{q}{x}} + \exp\bigg(-\frac{1}{3}\sqrt{\log x}\bigg) \bigg). $$
	\end{thm}
	\begin{proof}
		First of all, Granville and la Bret\`eche \cite[Equation 10.1]{granville breteche exp sums} proved that for a $1$-bounded completely multiplicative function $f$, one has
		$$\bigg|\sum_{\substack{n \leq x \\ P^+(n) \leq y}} f(n)e(\theta n)\bigg| \ll \sqrt{xy} + \bigg(\frac{x}{\sqrt{q}} + \sqrt{xq \log(2x/q)}\bigg)\log y + x e^{-(1+o(1))\sqrt{\log x  \log \log x}}.$$
		We take $f(n) = \bm{1}_{P^-(n) > z}$. In the case when $y \leq x/(\log x)^{2A}$, this implies the bound
		$$\bigg|\sum_{\substack{n \leq x \\ P^+(n) \leq y \\ P^-(n) > z}}e(\theta n)\bigg| \ll \frac{x}{(\log x)^{A}} + \bigg(\frac{x}{\sqrt{q}} + \sqrt{xq}\bigg)(\log x)^{3/2},$$
		where we have omitted the term $x\exp \{ -(1+o(1))\sqrt{\log x \log \log x}\}$ as it is negligible compared to $x/(\log x)^{A}$. It remains to consider the case when $y > x/(\log x)^{2A}$. Note that if $n\le x$ has $P^-(n)>z$ but is not $y$-smooth, then it can factored as $n=pm$ with $p>y$ and $p\ge P^+(m)$. Hence, we write $n = mp$, where $p = P^+(n) \in (y,x]$, $m \leq x/y$, and $P^+(m) \leq p$. This gives
		$$\sum_{\substack{n \leq x \\ P^+(n) \leq y \\ P^-(n) > z}}e(\theta n) =\sum_{\substack{n \leq x \\  P^-(n) > z}} e(\theta n) + O\bigg( \sum_{\substack{m \leq x/y}} \bigg| \sum_{\substack{y < p \leq x \\ P^+(m) < p}}e(\theta m p)\bigg|\bigg).$$ 
		We remark that the condition $P^+(m) < p$ can be omitted, since $P^+(m) \leq m \leq x/y < y < p$. As such, we can deduce from the above that
		\begin{equation}\label{eq: prime exp sum 1}
			\sum_{\substack{n \leq x \\ P^+(n) \leq y \\ P^-(n) > z}}e(\theta n) =\sum_{\substack{n \leq x \\  P^-(n) > z}} e(\theta n) + O\bigg( \sum_{\substack{m \leq x/y}} \bigg| \sum_{\substack{y < p \leq x/m}}e(\theta m p)\bigg|\bigg).
		\end{equation}
		To estimate the prime sum in the error term, we shall use \Cref{thm: exp sum primes}. Write $\alpha_m = \theta m$. We first check that $\alpha_m$ has a good rational approximation with
		denominator essentially $q$, up to a loss depending only on $m$. Put
		$g=(m,q)$. Since $(a,q)=1$, the fraction $am/q$ reduces to
		$$
		\frac{a_m}{q_m}:=\frac{am/g}{q/g},
		\qquad q_m=\frac{q}{g},
		\qquad (a_m,q_m)=1.
		$$
		Moreover,
		$$
		\left|\alpha_m-\frac{a_m}{q_m}\right|
		=
		m\left|\theta-\frac aq\right|
		\leq \frac{m}{q^2}
		=
		\frac{m/g^2}{q_m^2}.
		$$
		Thus $\alpha_m$ is within $B_m/q_m^2$ of a reduced fraction with
		denominator $q_m$, where
		$$
		B_m:=\frac{m}{g^2}\le m.
		$$
		We now replace this approximation by one of Dirichlet quality. Let
		$a'/q'$ be a best approximation to $\alpha_m$ among reduced fractions
		with denominator at most $q_m$. Since $a_m/q_m$ is one of the available
		fractions, we have
		$$
		\left|\alpha_m-\frac{a'}{q'}\right|
		\leq
		\left|\alpha_m-\frac{a_m}{q_m}\right|
		\leq
		\frac{B_m}{q_m^2}.
		$$
		If $a'/q'\ne a_m/q_m$, then
		$$
		\left|\frac{a'}{q'}-\frac{a_m}{q_m}\right|
		\geq \frac{1}{q'q_m},
		$$
		whereas the triangle inequality gives
		$$
		\left|\frac{a'}{q'}-\frac{a_m}{q_m}\right|
		\leq
		\left|\alpha_m-\frac{a'}{q'}\right|
		+
		\left|\alpha_m-\frac{a_m}{q_m}\right|
		\leq \frac{2B_m}{q_m^2}.
		$$
		Hence $q'\ge q_m/(2B_m)$. If instead $a'/q'=a_m/q_m$, then $q'=q_m$.
		Therefore in all cases
		$$
		q'\gg \frac{q_m}{B_m}
		=
		\frac{q/g}{m/g^2}
		=
		\frac{qg}{m}
		\geq \frac{q}{m},
		\qquad\text{and}\qquad
		q'\le q_m\le q.
		$$
		Finally, Dirichlet's Theorem implies that
		$$
		\left|\alpha_m-\frac{a'}{q'}\right|
		\le \frac{1}{(q')^2}.
		$$
		Thus, for every $m\le x/y$, there is a reduced fraction $a'/q'$ such
		that
		$$
		\left|\alpha_m-\frac{a'}{q'}\right|\le \frac{1}{(q')^2},
		\qquad
		\frac{q}{m}\ll q'\le q.
		$$
		Since $m\le x/y<(\log x)^{2A}$ in the present case, this gives
		$$
		q'\gg \frac{q}{(\log x)^{2A}},
		\qquad q'\le q.
		$$
		Returning to \eqref{eq: prime exp sum 1} and applying \Cref{thm: exp sum primes} with our $a',q'$, we find

		\begin{align} \nonumber
			\sum_{\substack{n \leq x \\ P^+(n) \leq y \\ P^-(n) > z}}e(\theta n) &=\sum_{\substack{n \leq x \\  P^-(n) > z}} e(\theta n) +  O\bigg(\sum_{\substack{m \leq x/y}} \bigg[\frac{x}{m} (\log x)^{2} \bigg(\sqrt{\frac{m}{q}+\frac{q}{x}} +\exp(-\frac{1}{2}\sqrt{\log x})\bigg)\bigg] +q \bigg) \\
			&= \sum_{\substack{n \leq x \\  P^-(n) > z}} e(\theta n) +O\bigg(x (\log x)^{2+A} \bigg(\sqrt{\frac{1}{q}+\frac{q}{x}} +\exp(-\frac{1}{2}\sqrt{\log x})\bigg) \bigg) + q (\log x)^{2A}. \label{eq: prime exp sum 2}
		\end{align}
		Applying \Cref{exp sum rough} to the sum in the above expression gives
		\begin{align} 	\sum_{\substack{n \leq x \\ P^+(n) \leq y \\ P^-(n) > z}}e(\theta n) \ll x (\log x)^{2+A} \bigg(\frac{1}{\sqrt{q}} +\frac{q}{\sqrt{x}} +\exp\bigg(-\frac{1}{3}\sqrt{\log x}\bigg)\bigg) .
		\end{align}
		This resolves the case $y > x/(\log x)^A,$ and hence concludes the proof of the theorem.
	\end{proof}

	\section{Probabilistic Foundations}
	In this section, we collect assorted results from probability theory which will be needed for our proofs. We begin by computing asymptotics for $\E|S_g(N)|^2$ and $\E|S_g(N;z)|^2$. We then move on to an overview of how martingale difference sequences can be applied to the study of random multiplicative functions.
	\subsection{Second Moment Calculations}
	Since $g$ is assumed to be a `nice' function, and $\theta$ is chosen such that the sequence $\{\theta n\}$ is uniformly distributed modulo 1, we can give simple asymptotics for $\E|S_g(N)|^2$ and $\E|S_g(N;z)|^2$ in terms of the $L^2$ norm of $g$.
	\begin{lem}\label{lem: variance estimate 1}
		Let $\theta \in \R\setminus \Q$ and $g \in  \CB\CV[\R/\Z]$ be given. Then one has
		$$V_{g}(N) = \sum_{n \leq N}\Big|g\big(\theta n\big)\Big|^2 \sim N \cdot \|g\|_2^2.$$
	\end{lem}
	\begin{proof}
		Since $g \in \CB\CV[\R/\Z]$, then so is $|g|^2$. In particular, $|g|^2$ is Riemann-integrable. Weyl's equidistribution theorem implies that the sequence $\{\theta n\}_{n \in \N}$ is uniformly distributed modulo 1, and so
		$$
		\sum_{n \leq N}\Big|g\big(\theta n\big)\Big|^2 \sim N \int_0^1 |g(t)|^2dt. \qedhere
		$$
	\end{proof}
	\begin{lem}\label{lem: variance estimate 2}
		Let $\theta \in \R\setminus \Q$ be given and suppose it satisfies \eqref{eq:sound-xu badness approx}. Let $g \in \CB\CV[\R/\Z]$ be given. Then, uniformly for $z \in [2,N/2]$, we have:
		$$V_{g}(N;z) = \sum_{\substack{n \leq N \\ P^-(n) >z}}\Big|g\big(\theta n\big)\Big|^2 \sim \Phi(N;z) \cdot \| g\|_2^2,$$
		where $\Phi(N;z) := \#\{n \leq N : P^-(n) > z\}$.
	\end{lem}
	\begin{proof}
		Write $H(t) = |g(t)|^2$. We will make use of Beurling-Selberg polynomials, which are sometimes called Vaaler polynomials in this context. Since $H$ is a periodic and non-negative function, Section 7 of \cite{vaaler} provides a proof that for each natural number $K$, there exist trigonometric polynomials $H_K^\pm$ of degree at most $K$, say $$H^\pm(t) = \sum_{|k| \leq K} c_k^\pm e(k t),$$
		such that:
		\begin{itemize}
			\item $H_K^-(t) \leq H(t) \leq H_K^+(t) \quad (t \in [0,1]),$
			\item $c_0^\pm = \|g\|_2^2 + O_g\big(K^{-1}\big),$
			\item $|c_k^\pm| \ll_g 1.$
		\end{itemize}
		Using these properties, we see that
		\begin{align*}
			\sum_{\substack{n \leq N \\ P^-(n) > z}} |g(\theta n)|^2 \leq \sum_{\substack{n \leq N \\ P^-(n) > z}} H^+(\theta n) 
			&= \sum_{\substack{n \leq N \\ P^-(n) > z}} \sum_{|k| \leq K} c_k^+  e(\theta n k) \\
			&= \bigg[\|g\|_2^2 + O_g\big(K^{-1}\big)\bigg]\sum_{\substack{n \leq N \\ P^-(n) > z}}1 + \sum_{0< |k| \leq K} c_k^+   \sum_{\substack{n \leq N \\ P^-(n) > z}} e(\theta n k).
		\end{align*}
		Set $K = \log N$, and we rewrite the above as
		\begin{align}\label{beurling 1}
			\sum_{\substack{n \leq N \\ P^-(n) > z}} |g(\theta n)|^2 \leq \Phi(N;z) \cdot \|g\|_2^2  + O_g\bigg(\frac{\Phi(N;z)}{K}\bigg) + O\bigg(\sum_{0< |k| \leq K}   \bigg|\sum_{\substack{n \leq N \\ P^-(n) > z}} e(\theta n k)\bigg| \bigg).
		\end{align}
		We find, using Dirichlet's approximation Theorem, integers $a,q$ such that $(a,q) = 1$, $1 \leq q \leq \sqrt{N}$, and $|\theta k - \frac{a}{q} | \leq \frac{1}{q \sqrt{N}}$. Thus, we have
		$$\|\theta k q\| \leq |\theta k q - a| \leq \frac{1}{\sqrt{N}}.$$
		By our diophantine assumption \eqref{eq:sound-xu badness approx} on $\theta$, we also have
		$$\|\theta k q\| \gg \exp(-| kq |^{1/127}).$$
		Combining these two facts gives $\frac{1}{\sqrt{N}} \gg \exp(- |K q|^{1/127} ),$ which can be rearranged as
		$$|q| \gg \frac{(\log N)^{127}}{K} \gg (\log N)^{126}.$$
		With our rational approximation to $\theta k$ in hand, we apply \Cref{exp sum rough} to each $n$ sum in \eqref{beurling 1}. This gives
		\begin{align}\label{beurling 1.5}
			\sum_{\substack{n \leq N \\ P^-(n) > z}} |g(\theta n)|^2 \leq \Phi(N;z) \cdot \|g\|_2^2  + O_g\bigg(\frac{\Phi(N;z)}{K}\bigg) + O\bigg( \frac{K \cdot N}{(\log N)^{126}}  \bigg).
		\end{align}
		Recalling that $K = \log N$ and using the fact that $\Phi(N;z) \gg N/ (\log N)$, the above simplifies as
		\begin{align*}
			\sum_{\substack{n \leq N \\ P^-(n) > z}} |g(\theta n)|^2 \leq \Phi(N;z) \cdot \|g\|_2^2 \cdot \bigg[ 1 + O_g\bigg(\frac{1}{\log N}\bigg)\bigg].
		\end{align*}
		By repeating this argument with the minorant $H_K^-$ in place of $H_K^+,$ one can similarly show that
		\begin{align*}
			\sum_{\substack{n \leq N \\ P^-(n) > z}} |g(\theta n)|^2 \ge \Phi(N;z) \cdot \|g\|_2^2 \cdot \bigg[ 1 + O_g\bigg(\frac{1}{\log N}\bigg)\bigg].
		\end{align*}
		Hence, one has
		\begin{align*}
			\sum_{\substack{n \leq N \\ P^-(n) > z}} |g(\theta n)|^2 = \Phi(N;z) \cdot \|g\|_2^2 \cdot \bigg[ 1 + O_g\bigg(\frac{1}{\log N}\bigg)\bigg] \sim \Phi(N;z) \cdot \|g\|_2^2.
		\end{align*}
		This proves the lemma.
	\end{proof}
	\subsection{Martingale Difference Sequences} \label{sec:MDS}
	Inspired in part by the work of Blei and Janson \cite{blei}, Harper \cite{harper2013limit} applied the theory of martingales in order to study random multiplicative functions. Here, we follow the philosophy which was later developed in \cite{harper2020moments}. Given a sequence $a_n$ and a prime number $p$, we write
	$$\CM_p(x) := \sum_{\substack{n \leq x \\ P^+(n) =p}}a_n f(n).$$
	Clearly, we have 
	$$\sum_{n \leq x} a_n f(n) = \sum_{p \leq x}\CM_p(x).$$ 
	Let $\mathcal F_p$ denote the sigma algebra generated by $\{f(q):q\le p\}$. That is:
	\[
	\mathcal F_p:=\sigma\big(\{f(q):q\le p\}\big).
	\]
	Then one can show that $\mathcal M_p(x)$ is a \textit{martingale difference sequence} with respect to the filtration $(\mathcal F_p)_p$. Explicitly, this means that $\mathcal{M}_p(x)$ is $\mathcal{F}_p$-measurable, integrable, and satisfies $$\mathbb{E}[\mathcal{M}_p(x) \mid \mathcal{F}_{<p}] = 0,$$ where $\mathcal{F}_{<p}$ denotes the sigma-algebra generated by $\{f(q) : q < p\}$. To see why the conditional expectation vanishes, note that any $n$ with largest prime factor $P^+(n) = p$ can be uniquely factored as $n = p^k m$, where $k \ge 1$ and $P^+(m) < p$. By the complete multiplicativity of $f$, we have $f(n) = f(p)^k f(m)$. Since $f(p)$ is independent of $\mathcal{F}_{<p}$ and distributed uniformly on the unit circle, we have
	$$ \mathbb{E}\big[f(p)^k f(m) \mid \mathcal{F}_{<p}\big] = f(m) \mathbb{E}\big[f(p)^k\big] = 0 $$
	for all $k \ge 1$, which linearly extends to give $\mathbb{E}[\mathcal{M}_p(x) \mid \mathcal{F}_{<p}] = 0$.

	This setup is fruitful, as it allows us to study the sum of random multiplicative functions by using theorems which apply to martingale difference sequences. Indeed, by applying McLeish's central limit theorem, Soundararajan and Xu established the following:\footnote{Note that we have modified the notation of the theorem as compared to how it appears in \cite{sound-xu-CLT}.}
	\begin{lem}[{Soundararajan \& Xu, 2023 \cite[Theorem 3.1]{sound-xu-CLT}}]\label{thm:sound-xu 3.1}
		Let $f$ denote a random Steinhaus multiplicative function, and let $a_n$ denote a sequence of complex numbers. Put
		$$V=V(N) := \sum_{n \leq N}|a_n|^2,$$
		and define the complex valued random variable 
		$$S=S(N):= \frac{1}{\sqrt{V(N)}}\sum_{n \leq N}a_n f(n).$$
		Suppose $\CS$ is a subset of $[2,N]$ such that for some $ \epsilon \in (0,1]$ the following three conditions hold:
		\noindent(1). We have
		$$\sum_{\substack{n \leq N \\ n \not \in \CS}}|a_n|^2 \leq \epsilon^2 V.$$
		(2). We have
		$$\bigg|\sum_{\substack{n_1,\dots,n_4 \in \CS \\ n_1n_2=n_3n_4 \\ n_1 \ne n_3,n_2 \ne n_4 \\ P^+(n_1)=P^+(n_3) \\ P^+(n_2) = P^+(n_4)}} a_{n_1}a_{n_2} \overline{a_{n_3}a_{n_4}}\bigg| \leq \epsilon^2 V^2.$$
		(3). We have 
		$$\bigg|\sum_{\substack{n_1,\dots,n_4 \in \CS \\ n_1n_2=n_3n_4 \\ P^+(n_1)=P^+(n_2) = P^+(n_3) = P^+(n_4)}}a_{n_1}a_{n_2} \overline{a_{n_3}a_{n_4}}\bigg| \leq \epsilon^4 V^2.$$
		Then for any real numbers $t_1$ and $t_2$ we have, with $t^2 = (t_1^2+t_2^2)/2$
		$$\E\big[e^{it_1 \Re(S(N))+it_2\Im(S(N))}\big] = e^{-t^2/2} +O(e^{t^2}\epsilon).$$
	\end{lem}
	In \Cref{sec:proof if}~we will apply \Cref{thm:sound-xu 3.1}~in order to prove the 'if' direction in \Cref{thm:A} as well as to prove \Cref{thm:B}. Conditions (1) and (3) turn out to be easily verified, and most of the work here goes into verifying condition (2).
	
	In \Cref{sec: proof only if}, we will avoid complicated combinatorics by using the following martingale inequality to bound high moments of the sum $|S_g(N)|$. 
	
	\begin{lem}[The Burkholder--Davis--Gundy Inequalities \cite{burkholder1972integral}]\label{thm:BDG}
		Let $(d_j)_{j \ge 1}$ be a complex-valued martingale difference sequence. For any $k \ge 1$, there exist positive constants $c_k$ and $C_k$ depending only on $k$ such that
		\[
		c_k \E\bigg[ \bigg( \sum_{j=1}^n |d_j|^2 \bigg)^k \bigg] \le \E\bigg[ \sup_{1 \le m \le n} \bigg| \sum_{j=1}^m d_j \bigg|^{2k} \bigg] \le C_k \E\bigg[ \bigg( \sum_{j=1}^n |d_j|^2 \bigg)^k \bigg].
		\]
	\end{lem}
	
	\section{Reduction to Mean-Zero Weights}\label{sec:reduction}
	In this section, we show that it suffices to prove \Cref{thm:A} under the assumption that $g \in \CB\CV_0[\R/\Z]$, and we formally extend \Cref{thm:B} to non-mean-zero weights under a restricted roughness threshold. 
	
	Let $g \in \CB\CV[\R/\Z]$ be a non-constant function and write $\Tilde{g} = g - \hat{g}(0)$, so that $\Tilde{g} \in \CB\CV_0[\R/\Z]$. Since the sum defining $\Delta$ in \eqref{eq: def Delta} only involves non-zero Fourier modes, we immediately have $\Delta(g) = \Delta(\Tilde{g})$. By \eqref{eq:harper better sqrt}, for any fixed constant $c$, we have
	\begin{equation}\label{eq: better than sqrt 1}
		\frac{c}{\sqrt{N}}\sum_{n \leq N} f(n) \xrightarrow[(N \to \infty)]{d} 0.
	\end{equation}
	Decomposing $g = \Tilde{g} + \hat{g}(0)$, we can write our normalized sum as
	\begin{equation}\label{eq: proof cor 1}
		S_{g}(N) = \sqrt{\frac{V_{\Tilde{g}}(N)}{V_{g}(N)}}\bigg(\frac{1}{\sqrt{V_{\Tilde{g}}(N)}}\sum_{n \leq N}\Tilde{g}(\theta n) f(n)\bigg) + \frac{\hat{g}(0)}{\sqrt{V_{g}(N)}}\sum_{n \leq N} f(n).
	\end{equation}
	
	By \Cref{lem: variance estimate 1}, we have $V_{g}(N) \sim N \|g\|_2^2$ and $V_{\Tilde{g}}(N) \sim N \|\Tilde{g}\|_2^2$. Therefore, the second term in \eqref{eq: proof cor 1} converges in distribution to $0$ by \eqref{eq: better than sqrt 1}. Assuming \Cref{thm:A} holds for the mean-zero function $\Tilde{g}$, the first sum in \eqref{eq: proof cor 1} converges in distribution to $\CC\CN(0,1)$ if and only if $\Delta(\Tilde{g}) = 0$. Applying Slutsky's theorem to \eqref{eq: proof cor 1} and using Parseval's identity ($\|\Tilde{g}\|_2^2 = \|g\|_2^2 - |\hat{g}(0)|^2$) yields
	\begin{equation*}
		S_{g}(N) \xrightarrow[(N \to \infty)]{d}  \sqrt{ \frac{\|\Tilde{g}\|_2^2}{\|g\|_2^2} }\CC\CN(0,1) = \CC\CN\bigg(0,1-\frac{|\hat{g}(0)|^2}{\|g\|_2^2}\bigg) \iff \Delta(g) = 0.
	\end{equation*}
	Conversely, if $\Delta(g) \ne 0$ and $S_g(N)$ converges to a complex normal random variable, the same Slutsky deduction implies the mean-zero sums for $\Tilde{g}$ would also converge to a complex normal random variable, contradicting \Cref{thm:A} for $\Tilde{g}$. For the proof of \Cref{thm:A}, we may therefore assume without loss of generality that $g \in \CB\CV_0[\R/\Z]$.
	
	The situation for \Cref{thm:B} is slightly more restrictive. To run the analogous Slutsky deduction over rough numbers, we require a substitute for \eqref{eq: better than sqrt 1}. This is provided by the following proposition, which is an immediate consequence of \cite[Theorem 1.2]{xu2024better}:
	
	\begin{prop}[{Follows from Xu, 2024 \cite[Theorem 1.2]{xu2024better}}]\label{prop better sqrt rough}
		Let $f$ be a Steinhaus random multiplicative function and $N$ be large. If $\log \log z = o(\sqrt{\log \log N}),$ then for any constant $c$, one has
		$$\frac{c}{\sqrt{N/\log z}}\sum_{\substack{n \leq N \\ P^-(n) > z}}f(n) \xrightarrow[(N \to \infty)]{d} 0.$$
	\end{prop}
	
	Because \Cref{prop better sqrt rough} imposes a strict growth condition on the roughness parameter, extending \Cref{thm:B} to general weights in $\CB\CV[\R/\Z]$ requires applying this same constraint. 
	
	\begin{cor}[Corollary to \Cref{thm:B}]\label{cor:B}
		Fix $\theta$ to be an irrational number satisfying \eqref{eq:sound-xu badness approx}. Let $g \in \CB\CV[\R/\Z]$ be a non-constant function. Let $\xi(N)$ be a function of $N$ which tends to $+\infty$, and satisfies $$\log \log \xi(N) = o(\sqrt{ \log \log N}).$$ Then $S_{g}(N;\xi(N))$ converges in distribution to a complex Gaussian random variable of mean $0$ and variance $1-|\hat{g}(0)|^2/\|g\|_2^2$ as $N \to \infty$. 
	\end{cor}
	\begin{proof}
		The proof follows exactly the same Slutsky's theorem deduction used above for \Cref{thm:A}. We write $\Tilde{g} = g - \hat{g}(0)$ and apply \Cref{lem: variance estimate 2} to handle the variance asymptotics. By substituting \Cref{prop better sqrt rough} in place of \eqref{eq: better than sqrt 1}, the $\hat{g}(0)$ term vanishes in distribution, yielding the scaled Gaussian limit. 
	\end{proof}

	\section{Proof of Central Limit Theorems}\label{sec:proof if}
	In this section we prove the `if' direction of \Cref{thm:A}, as well as \Cref{thm:B}. As was remarked in \Cref{sec:reduction}, we may assume $\hat{g}(0) = 0$. The proofs of both theorems follow a similar argument. We make the attempt to present both proofs in a unified way so as to avoid duplication of efforts. We define a roughness parameter $z \geq 1$ whose definition changes depending on which theorem we are proving:
	\begin{equation}\label{eq: def z}
		z := \begin{cases}
			1, \quad & \text{ in the context of \Cref{thm:A}},\\
			\xi(N), \quad & \text{ in the context of \Cref{thm:B}}
		\end{cases}.
	\end{equation}
	
	A density argument will show that it suffices to prove the results for a Fourier-truncated version of $g$. If $g \in \CB\CV_0[\R/\Z]$ is a non-constant function with Fourier coefficients $c_\l$, then we write
	\begin{equation}\label{eq: def trig poly}
		Q(t) = Q_{L,g}(t) := \sum_{|\l| \leq L} c_\l e(\l t), \quad (c_\l \ll_g |\l|^{-1}, c_0 = 0).
	\end{equation}
	The bound $c_\l \ll_g |\l|^{-1}$ holds uniformly for all $\l \in \Z$, and comes from the fact that $g \in \CB\CV[\R/\Z]$. In order to prove both \Cref{thm:A} and \Cref{thm:B}, we will apply \Cref{thm:sound-xu 3.1} with $\CS =  [2,N] \cap \{n : P^-(n) > z\},$ $z$ as in \eqref{eq: def z}, and $$a_n := \bm{1}_{P^-(n) > z} Q(\theta n).$$ 
	With this choice of $\CS,a_n$, it is clear that condition (1) of \Cref{thm:sound-xu 3.1} holds. Indeed, we trivially have
	$$\sum_{\substack{n \leq N \\ n \not \in \CS}} |a_n|^2 = \sum_{\substack{n \leq N \\ n \not \in \CS}} \big|\bm{1}_{P^-(n) > z} Q(\theta n)\big|^2 = 0.$$
	The next lemma shows that condition (3) is also satisfied as long as $N$ is sufficiently large in terms of $L$:
	\begin{lem}\label{lem: verif cond 3}
		Let $A > 0$ be given, let $\CS$ be as in \eqref{eq: def z} and let $Q$ be as in \eqref{eq: def trig poly}. Uniformly for $L \geq 1$, one has
		$$\bigg|\sum_{\substack{n_1,\dots,n_4 \in \CS \\ n_1n_2=n_3n_4 \\ P^+(n_1)=P^+(n_2) = P^+(n_3) = P^+(n_4)}}Q(\theta n_1)Q(\theta n_2)\overline{Q(\theta n_3)Q(\theta n_4)}\bigg| \ll_{A,g} \frac{N^2 (\log L)^4}{(\log N)^A}.$$
	\end{lem}
	\begin{proof}
		By expanding the definition of $Q$ using \eqref{eq: def trig poly} and applying the triangle inequality, we see that
		\begin{align}
			\label{eq: check cond 3 1} \bigg|&\sum_{\substack{n_1,\dots,n_4 \in \CS \\ n_1n_2=n_3n_4 \\ P^+(n_1)=P^+(n_2) = P^+(n_3) = P^+(n_4)}} Q(\theta n_1)Q(\theta n_2)\overline{Q(\theta n_3)Q(\theta n_4)}\bigg| \ll \\ &\ll \sum_{\substack{\l_1,\dots,\l_4 \\ |\l_i| \leq L}} \big|c_{\l_1}c_{\l_2}\overline{c_{\l_3}c_{\l_4}}\big| \sum_{\substack{n_1,\dots,n_4 \leq N  \\ n_1n_2=n_3n_4 \\ P^+(n_1)=\cdots= P^+(n_4)}} \prod_{i = 1}^4\bm{1}_{\CS}(n_i) \nonumber \\
			&\le \sum_{\substack{\l_1,\dots,\l_4 \\ |\l_i| \leq L}} \big|c_{\l_1}c_{\l_2}\overline{c_{\l_3}c_{\l_4}}\big| \sum_{\substack{n_1,\dots,n_4 \leq N  \\ n_1n_2=n_3n_4 \\ P^+(n_1)=\cdots= P^+(n_4)}}1, \nonumber
		\end{align}
		Note that each $n_i$ being summed can be rewritten as $n_i = pm_i$, where $p = P^+(n_1) = P^+(n_2) = P^+(n_3) = P^+(n_4),$ and $P^+(m_i) \leq p$. As such, we have
		\begin{align*}
			\sum_{\substack{n_1,\dots,n_4 \leq N  \\ n_1n_2=n_3n_4 \\ P^+(n_1)=\cdots= P^+(n_4)}}1 &= \sum_{p \leq N} \sum_{\substack{m_1,\dots,m_4 \leq N/p \\ m_1m_2=m_3m_4 \\ P^+(m_i) \leq p}}1\\
			&\le \sum_{p \leq N} \sum_{\substack{m_1,m_2 \leq N/p \\ P^+(m_i) \leq p}}\tau(m_1m_2).
		\end{align*}
		Since $\tau(m_1m_2)\le \tau(m_1)\tau(m_2)$ for all integers $m_1,m_2$, we can rewrite the above equation as
		\begin{align*}
			\sum_{\substack{n_1,\dots,n_4 \leq N  \\ n_1n_2=n_3n_4 \\ P^+(n_1)=\cdots= P^+(n_4)}}1 &\ll  \sum_{p \leq N} \bigg(\sum_{\substack{m \le N/p\\ P^+(m) \leq p}}\tau(m) \bigg)^2.
		\end{align*}
		By applying \Cref{lem: smooth divisor sums} to the above, we see that for a given $A > 0$, one has
		\begin{align}\label{eq: check cond 3 2}
			\sum_{\substack{n_1,\dots,n_4 \leq N  \\ n_1n_2=n_3n_4 \\ P^+(n_1)=\cdots= P^+(n_4)}}1 &\ll_A  \frac{N^2}{(\log N)^A}.
		\end{align}
		Substituting \eqref{eq: check cond 3 2} into \eqref{eq: check cond 3 1} yields
		\begin{align}
			\bigg|\sum_{\substack{n_1,\dots,n_4 \in \CS \\ n_1n_2=n_3n_4 \\ P^+(n_1)=P^+(n_2) = P^+(n_3) = P^+(n_4)}}Q(\theta n_1)Q(\theta n_2)\overline{Q(\theta n_3)Q(\theta n_4)}\bigg| &\ll_A \frac{N^2}{(\log N)^A}\sum_{\substack{\l_1,\dots,\l_4 \\ |\l_i| \leq L}} \big|c_{\l_1}c_{\l_2}\overline{c_{\l_3}c_{\l_4}}\big|.
		\end{align}
		Since $c_\ell \ll_g |\l|^{-1}$, the $L$ sum above is $\ll_g (\log L)^4$, uniformly in $L$.
	\end{proof}
	Having verified conditions (1) and (3) of \Cref{thm:sound-xu 3.1}, it remains to check that condition (2) is also satisfied. The following lemma is the first step in that direction. 
	
	\begin{lem}\label{lem:reduction cond 2} Let $N$ be large. Fix a number $\theta$ which satisfies \eqref{eq:sound-xu badness approx}. Let $\xi(N)$ be a function of $N$ which tends to $+\infty$ as $N \to \infty$. Let $\CS$ be defined as in \eqref{eq: def z}. Suppose $Q$ is a trigonometric polynomial of the form \eqref{eq: def trig poly}. Then, uniformly for $L \leq \min\{\xi(N),(\log N)^{1/10}\}$, one has
		
		\begin{align*}
			\sum_{\substack{n_1,\dots,n_4 \in \CS \\ n_1n_2=n_3n_4 \\ n_1 \ne n_3,n_2 \ne n_4 \\ P^+(n_1)=P^+(n_3) \\ P^+(n_2) = P^+(n_4)}}Q(\theta n_1)Q(\theta n_2)\overline{Q(\theta n_3)Q(\theta n_4)} &= \sum_{\substack{\l_1,\dots,\l_4 \\ |\l_i| \leq L}}C(\bm{\l}) \sum_{\substack{n_1,\dots,n_4 \in \CS \\ n_1n_2=n_3n_4 \\ \{n_1\l_1,n_2\l_2\}=\{n_3\l_3,n_4\l_4\} \\ n_1 \ne n_3,n_2 \ne n_4 \\ P^+(n_1)=P^+(n_3) \\ P^+(n_2) = P^+(n_4) }} 1\\& +O\bigg(\frac{N^2L}{(\log N)^{ \frac{4}{5}}}\bigg).
		\end{align*}
		where $C(\bm{\l}) := c_{\l_1}c_{\l_2}\overline{c_{\l_3}c_{\l_4}}$. 
	\end{lem}
	\begin{proof}[Proof of \Cref{lem:reduction cond 2}]
		By expanding the definition of $Q$, one has:
		\begin{equation}\label{eq:0 proof cond 2}
			\sum_{\substack{n_1,\dots,n_4 \in \CS \\ n_1n_2=n_3n_4 \\ n_1 \ne n_3,n_2 \ne n_4 \\ P^+(n_1)=P^+(n_3) \\ P^+(n_2) = P^+(n_4)}}Q(\theta n_1)Q(\theta n_2)\overline{Q(\theta n_3)Q(\theta n_4)} = \sum_{\substack{n_1,\dots,n_4 \in \CS \\ n_1n_2=n_3n_4 \\ n_1 \ne n_3,n_2 \ne n_4 \\ P^+(n_1)=P^+(n_3) \\ P^+(n_2) = P^+(n_4) }}\sum_{\substack{\l_1,\dots,\l_4 \\ |\l_i| \leq L}} C(\bm{\l}) e\Big[\theta \cdot \psi(\bm{n},\bm{\l})\Big],
		\end{equation}
		where $\psi(\bm{n},\bm{\l})	:= n_1\l_1+n_2\l_2-n_3\l_3-n_4\l_4.$
		
		Using the parameterization of \Cref{lem:parameterization}, we can write $n_1 = ga$, $n_2 = hb$, $n_3 = gb$, $n_4 = ha$, where $(a,b) = 1$ and $a \ne b$. We must determine how the conditions $P^+(n_1) = P^+(n_3)$ and $P^+(n_2) = P^+(n_4)$ constrain these variables. These equalities translate directly to $P^+(ga) = P^+(gb)$ and $P^+(ha) = P^+(hb)$. 
		
		Since $(a,b) = 1$ and $a \ne b$, the numbers $a$ and $b$ cannot share any prime factors, which guarantees that $P^+(a) \ne P^+(b)$ (if they were equal, they would both have to be $1$, contradicting $a \ne b$). Suppose, for the sake of contradiction, that $P^+(g) < \max\{P^+(a), P^+(b)\}$. Assuming without loss of generality that $P^+(a) > P^+(b)$, we would have $P^+(ga) = P^+(a)$. However, $P^+(gb) = \max\{P^+(g), P^+(b)\} < P^+(a)$, which contradicts $P^+(ga) = P^+(gb)$. Therefore, the largest prime factor of $ga$ and $gb$ must come from $g$, meaning we must have $P^+(a) \le P^+(g)$ and $P^+(b) \le P^+(g)$, which is equivalent to $P^+(ab) \le P^+(g)$. By identical reasoning, the condition $P^+(ha) = P^+(hb)$ implies $P^+(ab) \le P^+(h)$. 
		
		Vice versa, if $P^+(ab) \le \min\{P^+(g), P^+(h)\}$, then $P^+(ga) = \max\{P^+(a), P^+(g)\} = P^+(g)$ and $P^+(gb) = \max\{P^+(b), P^+(g)\} = P^+(g)$, so $P^+(ga) = P^+(gb)$. The exact same logic yields $P^+(ha) = P^+(hb)$. Finally, the overarching size constraint $n_i \le N$ trivially yields $\max(a,b) \times \max(g,h) \leq N$.
		
		The right hand side of \eqref{eq:0 proof cond 2} becomes
		\begin{equation}\label{eq lem cond 2: main sum to bound}
			\sum_{\substack{\l_1,\dots,\l_4 \\ |\l_i| \leq L}} C(\bm{\l}) \sum_{\substack{\max(a,b) \times \max(g,h) \leq N \\ a \ne b, (a,b) = 1 \\ P^+(ab) \leq \min\{P^+(g),P^+(h)\} \\ P^-(abgh) > z}} e\Big[\theta (ga\l_1+hb\l_2-gb\l_3-ha\l_4)\Big].
		\end{equation}
		The cases when $\{n_1\ell_1,n_2\ell_2\} =  \{n_3\ell_3,n_4\ell_4 \}$ correspond to the main term of the lemma. So it suffices to show that
		
		\begin{align*}
			\sum_{\substack{\l_1,\dots,\l_4 \\ |\l_i| \leq L}}C(\bm{\l}) \sum_{\substack{n_1,\dots,n_4 \in \CS \\ n_1n_2=n_3n_4 \\ \{n_1\l_1,n_2\l_2\}\ne\{n_3\l_3,n_4\l_4\} \\ n_1 \ne n_3,n_2 \ne n_4 \\ P^+(n_1)=P^+(n_3) \\ P^+(n_2) = P^+(n_4) }} 1 &\ll \frac{N^2L}{(\log N)^{ \frac{4}{5}}}.
		\end{align*}
		Thus, for the remainder of the proof, suppose $\bm{n},\bm{\l}$ are such that $\{n_1\l_1,n_2\l_2\} \ne \{n_3\l_3,n_4\l_4\}$. This is equivalent to the condition
		\begin{align}\label{eq lem cond 2: cond off diag}
			\{ga\l_1,hb\l_2\} \ne \{gb\l_3,ha\l_4\}.
		\end{align}
		The condition $\max(a,b) \times \max(g,h) \leq N$ implies that $\max(a,b) \leq \sqrt{N}$ or $\max(g,h) \leq \sqrt{N}$ or both. Following \cite{sound-xu-CLT}, we split into cases accordingly, making sure to subtract the double-count arising from when both conditions hold simultaneously.  
		
		\textbf{Case 1 ($\max(a,b) \leq \sqrt{N}$) : } We assume without loss of generality that $b < a$. The quantity to bound is
		\begin{equation}\label{eq:1 check prop 2 }
			\sum_{\substack{\l_1,\dots,\l_4 \\ |\l_i| \leq L}} C(\bm{\l}) \sum_{\substack{a \leq \sqrt{N} \\ b < a \\(a,b) = 1 \\ P^-(ab) > z }} \bigg( \sum_{\substack{g \leq N/a \\ P^+(g) \ge P^+(ab) \\P^-(g) > z}} e\Big[g\theta(a\l_1-b\l_3)\Big] \bigg) \bigg( \sum_{\substack{h \leq N/a \\ P^+(h) > P^+(ab) \\ P^-(h) > z}} e\Big[h\theta(b\l_2-a\l_4)\Big] \bigg).
		\end{equation}
		By the assumption \eqref{eq lem cond 2: cond off diag}, we see that $a\l_1 \ne b\l_3$ or $b\l_2 \ne a\l_4$. Assume without loss of generality that $a\l_1 \ne b\l_3$, as the other case is symmetrically analogous. Using the triangle inequality, we bound the $h$ sum in \eqref{eq:1 check prop 2 } by $N/a$, and are left with
		\begin{equation}\label{eq:2 check prop 2 }
			N\sum_{\substack{\l_1,\dots,\l_4 \\ |\l_i| \leq L}} |C(\bm{\l})| \sum_{\substack{a \leq \sqrt{N} \\ b < a,\;(a,b) = 1\\ a\l_1-b\l_3 \ne 0 \\ P^-(ab) > z}}\frac{1}{a} \cdot \bigg| \sum_{\substack{g \leq N/a \\ P^+(g) > P^+(ab) \\ P^-(g) > z}} e\Big[g\theta(a\l_1-b\l_3)\Big] \bigg|.
		\end{equation}
		In order to bound the $g$ sum, we split the range of $b$ into two parts. Let $\gamma > 0$ be a parameter to be optimized later, and define 
		$$\CB = \CB(a,\l_1,\l_3) = \bigg\{b \leq a : \exists q \in \Z_{\ne 0}, |q| \leq (\log N)^{\gamma} \text{ s.t. } \|\theta q (b\l_3 - a\l_1)\| \leq \frac{a(\log N)^{\gamma}}{N} \bigg\}.$$
		By applying \Cref{lem: badness spacing} with $X=a$, $Q=(\log N)^\gamma$, $s=\l_3$, $y=a\l_1$, and $\eta = a(\log N)^\gamma / N$, we find the spacing $\delta \gg (\log N)^{127-2\gamma}/L$. Thus we have
		\begin{equation} \label{bounding CB1}
			|\CB| \ll 1+\frac{a \cdot L}{(\log N)^{127 - 2\gamma }}.
		\end{equation}
		The `$+1$' in the above bound is problematic, but it can be eliminated by noting that $b$ (and hence $a$) is not too small. Indeed, since $b \in \CB$, we know that $\|\theta q(a\l_1 - b\l_3)\| \leq \frac{a (\log N)^\gamma}{N} \ll N^{-1/4}$. Combining this with the diophantine condition \eqref{eq:sound-xu badness approx} implies $N^{-1/4} \gg \exp(- |q(a \l_1 - b\l_3)|^{1/127}).$ Taking logs and rearranging the inequality gives (by the triangle inequality):
		$$(\log N)^{127} \ll |q(a \l_1-b\l_3)|  \leq 2 L (\log N)^\gamma \max\{a,b\}.$$
		Since $b \leq a$, we thus have $a \geq (\log N)^{127-\gamma - 1/9}.$
		Applying this to \eqref{bounding CB1}, we see that so long as $\gamma \geq 2$, we have
		\begin{equation} \label{bounding CB}
			|\CB| \ll \frac{a \cdot L}{(\log N)^{127 - 2\gamma }},
		\end{equation}
		without the `$+1$' term. For each $b \in \CB$, we use the triangle inequality to bound the sum over $g$ in \eqref{eq:2 check prop 2 } by $N/a$. We see that \eqref{eq:2 check prop 2 } is
		\begin{equation}\label{eq:5 check prop 2 }
			\ll N\sum_{\substack{\l_1,\dots,\l_4 \\ |\l_i| \leq L}} |C(\bm{\l})| \sum_{\substack{a \leq \sqrt{N}}}\frac{1}{a}\bigg(\sum_{\substack{b < a,\; b \not\in \CB \\  a\l_1-b\l_3 \ne 0 }} \bigg| \sum_{\substack{g \leq N/a \\ P^+(g) > P^+(ab) \\ ga \in  \CS}} e\Big[g\theta(a\l_1-b\l_3)\Big] \bigg| +\sum_{\substack{b < a \\ b \in \CB }} \frac{N}{a}   \bigg).
		\end{equation}
		By \eqref{bounding CB} we have
		\begin{equation}\label{eq:6 check prop 2 }
			\sum_{\substack{b < a \\ b \in \CB}} \frac{N}{a}  \leq \frac{N}{a}|\CB| \ll \frac{N\cdot L}{(\log N)^{127 - 2\gamma}}.
		\end{equation}
		To estimate the exponential sum over $g$ when $b \not \in \CB$, we proceed as follows. Let $a \leq \sqrt{N},\l_1,\l_3, b \not \in \CB$ be given. Using Dirichlet's approximation theorem, find a rational number $r/q$ with $(r,q) = 1 $ and $1 \leq q \leq \frac{N}{a(\log N)^{\gamma}}$ such that $|\theta (a\l_1-b\l_3) - r/q| \leq a(\log N)^{\gamma}/(q N)$. Since $b \not \in \CB$, it must be the case that $q > (\log N)^\gamma$. We can rewrite the $g$ sum in \eqref{eq:5 check prop 2 } as
		\begin{align*}
			\sum_{\substack{g \leq N/a \\ P^+(g) > P^+(ab) \\ P^-(g) > z}} e\Big[g\theta(a\l_1-b\l_3)\Big] &=\bm{1}_{P^-(ab) > z}\sum_{\substack{g \leq N/a \\ P^+(g) > P^+(ab) \\ P^-(g) > z}} e\Big[g\theta(a\l_1-b\l_3)\Big]   \\ &= \bm{1}_{P^-(ab) > z}\bigg(\sum_{\substack{g \leq N/a \\ P^-(g) > z}} e\Big[g\theta(a\l_1-b\l_3)\Big]-\sum_{\substack{g \leq N/a \\ P^+(g) \leq P^+(ab) \\P^-(g) > z}} e\Big[g\theta(a\l_1-b\l_3)\Big]\bigg). 
		\end{align*}
		Thus, by the triangle inequality applied to the above, we have
		\begin{align}\label{eq:exp sum estimate pre sieve}
			\sum_{\substack{g \leq N/a \\ P^+(g) > P^+(ab) \\ P^-(g) > z}} e\Big[g\theta(a\l_1-b\l_3)\Big]&\ll \bigg|\sum_{\substack{g \leq N/a \\ P^+(g) \leq P^+(ab) \\P^-(g) > z}} e\Big[g\theta(a\l_1-b\l_3)\Big] \bigg|+ \bigg|\sum_{\substack{g \leq N/a \\ P^-(g) > z}} e\Big[g\theta(a\l_1-b\l_3)\Big]\bigg|
		\end{align}
		The two sums of the right hand side of \eqref{eq:exp sum estimate pre sieve} are exponential sums over smooth numbers with a roughness condition (in the case $z=1$ the roughness condition is trivial). We apply \Cref{thm: exp sum smooth} with $A = 2$ to each sum, and we see that  
		\begin{align*}
			\bigg|\sum_{\substack{g \leq N/a \\ P^+(g) > P^+(ab) \\ P^-(g) > z}} e\Big[g\theta(a\l_1-b\l_3)\Big]\bigg| &\ll \bigg|\sum_{\substack{g \leq N/a \\ P^+(g) \leq P^+(ab) \\ P^-(g) > z}} e\Big[g\theta(a\l_1-b\l_3)\Big]\bigg|+\bigg|\sum_{\substack{g \leq N/a \\ P^-(g) > z}} e\Big[g\theta(a\l_1-b\l_3)\Big]\bigg| \\ 
			\nonumber &\ll \frac{N}{a (\log N)^2} + \frac{N}{a} (\log N)^{4} \bigg(\frac{1}{q^{1/2}} + \frac{(qa)^{1/2}}{{N}^{1/2}}\bigg).
		\end{align*}
		By our assumptions on $a,b$ (namely that $a \leq \sqrt{N}, q > (\log N)^\gamma, |\theta(a\l_1-b\l_3) - r/q| \leq a(\log N)^{\gamma}/(q N)$), the above can be simplified as
		\begin{align}\label{eq:exp sum estimate case AB}
			\bigg|\sum_{\substack{g \leq N/a \\ P^+(g) > P^+(ab) \\ P^-(g) > z}} e\Big[g\theta(a\l_1-b\l_3)\Big]\bigg| &\ll \frac{N}{a (\log N)^{\min\{2,\gamma/2 - 4\}}}.
		\end{align}
		
		Substituting  \eqref{eq:6 check prop 2 } and \eqref{eq:exp sum estimate case AB} into \eqref{eq:5 check prop 2 }, we see that \eqref{eq:2 check prop 2 } is 
		\begin{equation}\label{eq:8 check prop 2 }
			\ll \frac{N^2 \cdot L}{(\log N)^{\min\{2,127-2\gamma,\gamma/2-4\}}}\sum_{\substack{\l_1,\dots,\l_4 \\ |\l_i| \leq L}} |C(\bm{\l})| \sum_{\substack{a \leq \sqrt{N} }}\frac{1}{a} \ll \frac{N^2 \cdot L}{(\log N)^{\min\{2,127-2\gamma,\gamma/2-4\} -1}}\sum_{\substack{\l_1,\dots,\l_4 \\ |\l_i| \leq L}} |C(\bm{\l})|.
		\end{equation}
		Since $c_{\ell} \ll |\l|^{-1},$ we have $\sum_{\bm{\l} :|\l_i| \leq L} |C(\bm{\l})| \ll (\log L)^4,$ uniformly in $L$. Substituting this bound for the $\ell$ sum into \eqref{eq:8 check prop 2 }, we see that
		\begin{equation}\label{eq:9 check prop 2 }
			N\sum_{\substack{\l_1,\dots,\l_4 \\ |\l_i| \leq L}} |C(\bm{\l})| \sum_{\substack{a \leq \sqrt{N} \\ b < a,\;(a,b) = 1\\ a\l_1-b\l_3 \ne 0 \\ P^-(ab) > z}}\frac{1}{a} \cdot \bigg| \sum_{\substack{g \leq N/a \\ P^+(g) > P^+(ab) \\ P^-(g) > z}} e\Big[g\theta(a\l_1-b\l_3)\Big] \bigg| \ll \frac{N^2 \cdot L (\log L)^4}{(\log N)^{\min\{2,127-2\gamma,\gamma/2-4\} -1}}.
		\end{equation}
		By taking $\gamma = 13,$ we conclude Case 1, as we have shown
		$$\sum_{\substack{\l_1,\dots,\l_4 \\ |\l_i| \leq L}} C(\bm{\l}) \sum_{\substack{a \leq \sqrt{N} \\ b < a \\  (a,b) = 1 \\ P^-(ab) > z}} \bigg( \sum_{\substack{g \leq N/a \\ P^+(g) > P^+(ab) \\P^-(g) > z}} e\Big[g\theta(a\l_1-b\l_3)\Big] \bigg) \bigg( \sum_{\substack{h \leq N/a \\ P^+(h) > P^+(ab) \\ P^-(h) > z }} e\Big[h\theta(b\l_2-a\l_4)\Big] \bigg) \ll \frac{N^2\cdot L (\log L)^4}{(\log N)}.$$

		\textbf{Case 2 ($\max(g,h) \leq \sqrt{N}$) : } Assume without loss of generality that $h < g$. The quantity to bound is
		\begin{equation}\label{eq:10 check prop 2 }
			\sum_{\substack{\l_1,\dots,\l_4 \\ |\l_i| \leq L}} C(\bm{\l}) \sum_{\substack{g \leq \sqrt{N} \\ h < g \\ P^-(gh) > z }}  \sum_{\substack{a,b \leq N/g \\ (a,b) = 1, a \ne b  \\  P^+(a),P^+(b) \leq  \min\{P^+(g),P^+(h)\}  \\ P^-(ab) > z}} e\Big[a\theta(g\l_1-h\l_4)+b\theta(h\l_2-g\l_3)\Big] ,
		\end{equation}
		and we wish to bound this under the assumption that $\{n_1\l_1,n_2\l_2\} \ne \{n_3\l_3,n_4\l_4\}$. By using M\"obius inversion to detect the condition $(a,b) = 1$, the sums over $a,b$ in \eqref{eq:10 check prop 2 } can be written as
		\begin{equation}\label{eq:11 check prop 2}
			\sum_{\substack{k \leq N/g \\ P^+(k) \leq \min\{P^+(g),P^+(h)\} \\ P^-(k) >z}}\mu(k) \sum_{\substack{u,v \leq N/(kg) \\ P^+(u),P^+(v) \leq \min\{P^+(g),P^+(h)\}\\ P^-(uv) > z}}e\big(\theta k [u(g \l_1 - h\l_4) + v(h\l_2-g\l_3)]\big).
		\end{equation}
		By our assumption $\{n_1\l_1,n_2\l_2\} \ne \{n_3\l_3,n_4\l_4\}$, we see that $g\l_1 \ne h\l_4$ or $h\l_2 \ne g\l_3$. Assume that $g\l_1 \ne h\l_4$, as the other case is analogous.
		
		If $k > (\log N)^{127},$ then we use the triangle inequality to bound both exponential sums in \eqref{eq:11 check prop 2} by $N/(kg)$, and get a contribution of
		\begin{equation}\label{eq:12 check prop 2}
			\ll \sum_{\substack{k>(\log N)^{127} \\ P^+(k) \leq \min\{P^+(g),P^+(h)\} \\ P^-(k) >z }}\frac{N^2}{k^2 g^2 } \ll \sum_{\substack{k>(\log N)^{127}}}\frac{N^2}{k^2 g^2 } \ll \frac{N^2}{g^2 (\log N)^{127}}.
		\end{equation}
		
		If $k \leq (\log N)^{127}$, we use the triangle inequality to bound the $v$ sum in \eqref{eq:11 check prop 2} by $N/(kg)$, and are left with
		\begin{equation}\label{eq:13 check prop 2}
			\frac{N}{g}\sum_{\substack{k \leq (\log N)^{127} \\ P^+(k) \leq \min\{P^+(g),P^+(h)\} \\ P^-(k) > z}} \frac{1}{k} \cdot \bigg|\sum_{\substack{u \leq N/(kg) \\ P^+(u) \leq \min\{P^+(g),P^+(h)\} \\ P^-(u) > z}}e\big(\theta k u(g \l_1 - h\l_4)\big)\bigg|.
		\end{equation}
		Define $$\CH = \CH(g,\l_1,\l_4,k) := \bigg\{h \leq g : \exists q \in \Z_{\ne 0}, |q| \leq (\log (N/g))^{\gamma} \text{ s.t. } \|\theta k q (g \l_1 - h\l_4)\| \leq \frac{kg(\log N)^{\gamma}}{N}\bigg\},$$
		where $\gamma > 0$ is a parameter to be optimized later. In the case when $h \not \in \CH$, we use Dirichlet's approximation to find a rational number $r/q$ with $(r,q) = 1$ and $1 \leq q \leq N/(kg(\log N)^\gamma)$ such that $|\theta k (g \l_1 - h\l_4)-r/q| \leq kg(\log N)^\gamma/(qN)$. By repeating the argument from Case 1 which led up to \eqref{eq:exp sum estimate case AB} (i.e. applying \Cref{thm: exp sum smooth}), one can show that
		\begin{equation}\label{eq:13.5 check prop 2}
			\bigg|\sum_{\substack{u \leq N/(kg) \\ P^+(u) \leq \min\{P^+(g),P^+(h)\} \\ P^-(u) > z}}e\big(\theta k u(g \l_1 - h\l_4)\big)\bigg| \ll \frac{N}{kg (\log N)^{\min\{\gamma/2-4,2\}}}.
		\end{equation}
		On the other hand, when $h \in \CH$, we use the triangle inequality to bound the exponential sum in \eqref{eq:13 check prop 2} by $N/(kg)$. Combining the bounds in the two cases $h \in \CH$ and $h \not \in \CH$, we see that \eqref{eq:13 check prop 2} can be bounded as
		\begin{align*}
			&\ll \frac{N^2}{g^2} \sum_{\substack{k \leq (\log N)^{127} \\ P^+(k) \leq \min\{P^+(g),P^+(h)\} \\ P^-(k) > z}} \frac{1}{k^2}\bigg(\bm{1}_{h \in \CH} + \bm{1}_{h \not \in \CH} \frac{L}{(\log N)^{\min\{\gamma/2-4,2\}}}\bigg) \\ 
			&\ll \frac{N^2}{g^2}\sum_{\substack{k \leq (\log N)^{127}}} \frac{1}{k^2} \bigg(\bm{1}_{h \in \CH} + \bm{1}_{h \not \in \CH} \frac{L}{(\log N)^{\min\{\gamma/2-4,2\}}}\bigg).
		\end{align*}
		Using this bound, we see that \eqref{eq:10 check prop 2 } can be bounded by
		\begin{align*}
			&N^2\sum_{\substack{\l_1,\dots,\l_4 \\ |\l_i| \leq L}} |C(\bm{\l})| \sum_{\substack{k \leq (\log N)^{127}}} \frac{1}{k^2} \sum_{\substack{g \leq \sqrt{N}}} \frac{1}{g^2} \bigg(\sum_{\substack{h < g \\ h \in \CH }}1 +\sum_{\substack{h < g \\ h \not\in \CH}}\frac{L}{(\log N)^{\min\{\gamma/2-4,2\}}}\bigg)\\
			&\ll N^2\sum_{\substack{\l_1,\dots,\l_4 \\ |\l_i| \leq L}} |C(\bm{\l})|\sum_{\substack{k \leq (\log N)^{127}}} \frac{1}{k^2} \sum_{\substack{g \leq \sqrt{N}}} \frac{1}{g^2} \bigg( |\CH| +\frac{g \cdot L}{(\log N)^{\min\{\gamma/2-4,2\}}}\bigg).
		\end{align*}
		
		Applying \Cref{lem: badness spacing} with $X=g$, $Q=(\log(N/g))^\gamma$, $s=k\l_4$, $y=gk\l_1$, and $\eta = kg(\log N)^\gamma/N$, we obtain a spacing $\delta \gg (\log N)^{127-2\gamma} / (k L)$. This yields the cardinality bound $|\CH| \ll 1+g k L/(\log N)^{127 - 2\gamma}$. Arguing as we did in Case 1, we can eliminate the `$+1$' from the bound on $|\CH|$. As such, the above is
		\begin{align*}
			&\ll \frac{N^2\cdot L}{(\log N)^{\min\{2,\gamma/2-4,127-2\gamma\}}} \sum_{\substack{\l_1,\dots,\l_4 \\ |\l_i| \leq L}} |C(\bm{\l})| \sum_{\substack{k \leq (\log N)^{127}}} \frac{1}{k}\sum_{\substack{g \leq \sqrt{N} }} \frac{1}{g}\\
			&\ll \frac{N^2 (\log \log N)\cdot L}{(\log N)^{\min\{2,\gamma/2-4,127-2\gamma\}- 1}} \sum_{\substack{\l_1,\dots,\l_4 \\ |\l_i| \leq L}} |C(\bm{\l})|.
		\end{align*}
		We use the fact that $c_\l \ll |\l|^{-1}$ to bound the $\l$ sum by $(\log L)^4$. Again taking $\gamma = 13$, we see that
		\begin{equation}\label{eq:16 check prop 2 }
			\frac{N^2(\log \log N)\cdot L}{(\log N)^{\min\{2,\gamma/2-4,127-2\gamma\}- 1}} \sum_{\substack{\l_1,\dots,\l_4 \\ |\l_i| \leq L}} |C(\bm{\l})| \ll \frac{N^2(\log \log N) L (\log L)^4}{(\log N)}.
		\end{equation}
		This concludes Case 2.
		
		\textbf{Case 3 ($\max(g,h),\max(a,b) \leq \sqrt{N}$) : }One can argue exactly as in Case 1 and get the bound
		\begin{equation*}
			\sum_{\substack{\l_1,\dots,\l_4 \\ |\l_i| \leq L}} C(\bm{\l}) \sum_{\substack{\max(a,b) , \max(g,h) \leq N \\ a \ne b, (a,b) = 1 \\ P^+(ab) \leq \min\{P^+(g),P^+(h)\} \\ ga,ha,gb,hb \in \CS}} e\Big[\theta (ga\l_1+hb\l_2-gb\l_3-ha\l_4)\Big] \ll \frac{N^2 L (\log L)^4}{(\log N)}.
		\end{equation*}
		\textbf{Summary of the 3 cases.}
		Until now, we have shown that the contribution to \eqref{eq:0 proof cond 2} from terms with $\{n_1\l_1,n_2\l_2\} \ne \{n_3\l_3,n_4\l_4\}$ is  $O(N^2 L (\log L)^4 (\log \log N) (\log N)^{-1})$. As such, \eqref{eq:0 proof cond 2} may be rewritten as 
		\begin{equation*}
			\sum_{\substack{\l_1,\dots,\l_4 \\ |\l_i| \leq L}} C(\bm{\l}) \sum_{\substack{n_1,\dots,n_4 \in \CS \\ n_1n_2=n_3n_4 \\ \{n_1\l_1,n_2\l_2\}=\{n_3\l_3,n_4\l_4\} \\ n_1 \ne n_3,n_2 \ne n_4 \\ P^+(n_1)=P^+(n_3) \\ P^+(n_2) = P^+(n_4) }} 1 + O\bigg(\frac{N^2 (\log \log N) L (\log L)^4}{(\log N)}\bigg).
		\end{equation*}
		Since we've assumed that $L \ll (\log N)^{1/10},$ we can bound the factor $(\log \log N) (\log L)^4$ by $(\log N)^{1/5}$ in the above error term. This proves the lemma.
	\end{proof}
	The next lemma simplifies the complicated sum which results from \Cref{lem:reduction cond 2}.
	
	\begin{lem}\label{lem:verif cond 2 thm A} Let $N$ be large in terms of $\theta$. Suppose $Q_L$ is a trigonometric polynomial of the form \eqref{eq: def trig poly}, and let $z \in \{1, \xi(N)\}$. Then, uniformly for $L \leq (\log N)^{1/10}$, one has
		\begin{align*}
			\sum_{\substack{n_1,\dots,n_4 \leq N \\ P^-(n_i) > z \\ n_1n_2=n_3n_4 \\ n_1 \ne n_3,n_2 \ne n_4 \\ P^+(n_1)=P^+(n_3) \\ P^+(n_2) = P^+(n_4)}} \hspace{-10pt} Q(\theta n_1)Q(\theta n_2)\overline{Q(\theta n_3)Q(\theta n_4)} = \begin{cases}
				N^2 \Delta(Q_L) + O\bigg(\frac{N^2L}{(\log N)^{\frac{4}{5}}}\bigg) & \text{if } z=1, \\
				O\bigg(\frac{N^2L}{(\log N)^{\frac{4}{5}}}\bigg) & \text{if } z=\xi(N).
			\end{cases}
		\end{align*}
		where $\Delta$ is defined in \eqref{eq: def Delta}.
	\end{lem}
	\begin{proof}
		First of all, by the definition of $Q$ and \Cref{lem:reduction cond 2}, we have
		\begin{equation}\label{eq: verif cond 2 first}
			\sum_{\substack{n_1,\dots,n_4 \in \CS \\ n_1n_2=n_3n_4 \\ n_1 \ne n_3,n_2 \ne n_4 \\ P^+(n_1)=P^+(n_3) \\ P^+(n_2) = P^+(n_4)}}Q(\theta n_1)Q(\theta n_2)\overline{Q(\theta n_3)Q(\theta n_4)}= \sum_{\substack{\l_1,\dots,\l_4 \\ |\l_i| \leq L}} C(\bm{\l}) \sum_{\substack{n_1,\dots,n_4 \in \CS \\ n_1n_2=n_3n_4 \\ \{n_1\l_1,n_2\l_2\}=\{n_3\l_3,n_4\l_4\} \\ n_1 \ne n_3,n_2 \ne n_4 \\ P^+(n_1)=P^+(n_3) \\ P^+(n_2) = P^+(n_4) }} \hspace{-1cm} 1 +O\bigg(\frac{N^2 L}{(\log N)^{\frac{4}{5}}}\bigg).
		\end{equation}
		We will simplify the inner sum on the right hand side, and then sum over $\l$. By an application of Rankin's trick, one can show that for any $A>0$,
		\begin{equation}\label{eq: replace Pn Pnl}
			\sum_{\substack{n_1,\dots,n_4 \leq N \\ n_1n_2 = n_3n_4 \\ \exists i : P^+(n_i) \leq P^+(\l_i)}} 1 \ll \sum_{\substack{n_1,n_2 \leq N \\ P^+(n_1) \leq L}} \tau(n_1n_2) \ll_A \frac{N^2}{(\log N)^A}.
		\end{equation}
		As such, we may assume throughout the course of our calculations that $P^+(n_i) = P^+(n_i\l_i)$ for $1 \leq i \leq 4$, at the cost of an error term $O(N^2/(\log N)^A)$. In particular, we have
		\begin{equation}\label{eq:intermed 1} 
			\sum_{\substack{n_1,\dots,n_4 \in \CS \\ n_1n_2=n_3n_4 \\ \{n_1\l_1,n_2\l_2\}=\{n_3\l_3,n_4\l_4\} \\ n_1 \ne n_3,n_2 \ne n_4 \\ P^+(n_1)=P^+(n_3) \\ P^+(n_2) = P^+(n_4)}} 1 = \sum_{\substack{n_1,\dots,n_4 \in \CS \\ n_1n_2=n_3n_4 \\ \{n_1\l_1,n_2\l_2\}=\{n_3\l_3,n_4\l_4\} \\ n_1 \ne n_3,n_2 \ne n_4 \\ P^+(n_1)=P^+(n_1 \l_1)=P^+(n_3 \l_3) = P^+(n_3) \\ P^+(n_2) = P^+(n_2 \l_2) = P^+(n_4 \l_4) = P^+(n_4) }} \hspace{-1cm} 1 + O_A\Big(\frac{N^2}{(\log N)^A}\Big).
		\end{equation}
		We decompose the sum on the right hand side of \eqref{eq:intermed 1} into two parts, and subtract a double-counting:
		\begin{equation}\label{eq: S1 S2 decomp}
			\sum_{\substack{n_1,\dots,n_4 \in \CS \\ n_1n_2=n_3n_4 \\ \{n_1\l_1,n_2\l_2\}=\{n_3\l_3,n_4\l_4\} \\ n_1 \ne n_3,n_2 \ne n_4 \\ P^+(n_1)=P^+(n_1 \l_1)=P^+(n_3 \l_3) = P^+(n_3) \\ P^+(n_2) = P^+(n_2 \l_2) = P^+(n_4 \l_4) = P^+(n_4) }} 1 = S_1 + S_2 - S_3,
		\end{equation}
		where
		\begin{align*}
			S_1 &:= \sum_{\substack{n_1,\dots,n_4 \in \CS \\ n_1n_2=n_3n_4 \\ n_1\l_1=n_3\l_3, n_2\l_2 = n_4\l_4 \\ n_1 \ne n_3,n_2 \ne n_4 \\ P^+(n_1) = P^+(n_1 \l_1)=P^+(n_3 \l_3) = P^+(n_3) \\P^+(n_2) =  P^+(n_2 \l_2) = P^+(n_4 \l_4) = P^+(n_4)}} 1; \quad S_2 := \sum_{\substack{n_1,\dots,n_4 \in \CS \\ n_1n_2=n_3n_4 \\ n_1\l_1=n_4\l_4, n_2\l_2 = n_3\l_3 \\ n_1 \ne n_3,n_2 \ne n_4 \\ P^+(n_1) = P^+(n_1 \l_1)=P^+(n_3 \l_3) = P^+(n_3) \\P^+(n_2) =  P^+(n_2 \l_2) = P^+(n_4 \l_4) = P^+(n_4)}} 1. \\ \\
			S_3 &:= \sum_{\substack{n_1,\dots,n_4 \in \CS \\ n_1n_2=n_3n_4 \\ n_1\l_1=n_3\l_3= n_2\l_2 = n_4\l_4 \\ n_1 \ne n_3,n_2 \ne n_4 \\ P^+(n_1) = P^+(n_1 \l_1)=P^+(n_3 \l_3) = P^+(n_3) \\P^+(n_2) =  P^+(n_2 \l_2) = P^+(n_4 \l_4) = P^+(n_4)}} 1.
		\end{align*}
		
		The sum $S_3$ is negligible compared to $S_1,S_2$: By omitting many conditions from the $S_3$ sum, and merely using the fact that fixing one of the $n_i$ gives at most $2L$ choices for the other $n_j$, one has
		\begin{align}\label{eq:new S3 bound}
			S_3 \ll \sum_{\substack{n_1,\dots,n_4 \in \CS \\ n_1\l_1=n_3\l_3= n_2\l_2 = n_4\l_4}} 1 \ll N \cdot L^4.
		\end{align}

		The conditions on $n_i$ in the sum $S_2$ imply that $P^+(n_1) = P^+(n_2) = P^+(n_3) = P^+(n_4)$. Thus, by \Cref{lem: verif cond 3}, we know that
		\begin{equation}\label{eq: S2 bound}
			S_2 \ll_A \frac{N^2}{(\log N)^A}.
		\end{equation}
		It remains to simplify $S_1$. For a given $(\l_1,\dots,\l_4)$, the numbers $(n_1,\dots,n_4)$ being summed in $S_1$ must satisfy $n_1n_2=n_3n_4, n_1 \ne n_3, n_2 \ne n_4$. Because of the conditions $n_1\l_1 = n_3\l_3, n_2\l_2 = n_4\l_4,$ this is equivalent to $\l_1\l_2=\l_3\l_4, \l_1 \ne \l_3, \l_2 \ne \l_4$. Thus, we can rewrite $S_1$ as
		\begin{align*}
			S_1 = \bm{1}_{\substack{\l_1\l_2=\l_3\l_4 \\ \l_1 \ne \l_3, \l_2 \ne \l_4}} \times \sum_{\substack{n_1,\dots,n_4 \in \CS  \\ n_1\l_1=n_3\l_3, n_2\l_2 = n_4\l_4\\ P^+(n_1) = P^+(n_1 \l_1)=P^+(n_3 \l_3) = P^+(n_3) \\P^+(n_2) =  P^+(n_2 \l_2) = P^+(n_4 \l_4) = P^+(n_4)}} \hspace{-1cm} 1.
		\end{align*}
		In the above sum, $n_1,n_3$ are decoupled from $n_2,n_4$, and so we can split the sum into a product:
		\begin{align}\label{eq: S1 halfway reduction}
			S_1 = \bm{1}_{\substack{\l_1\l_2=\l_3\l_4 \\ \l_1 \ne \l_3, \l_2 \ne \l_4}} \times \bigg( \sum_{\substack{n_1,n_3 \in \CS \\ n_1\l_1=n_3\l_3  \\ P^+(n_1) = P^+(n_1 \l_1)=P^+(n_3 \l_3) = P^+(n_3)}} 1\bigg)\bigg( \sum_{\substack{n_2,n_4 \in \CS \\ n_2\l_2 = n_4\l_4 \\ P^+(n_2) = P^+(n_2 \l_2)=P^+(n_4 \l_4) = P^+(n_4)}} 1\bigg).
		\end{align}
		We note that
		\begin{equation} \label{eq: S1 prime factor reduction}
			\sum_{\substack{n_1,n_3 \in \CS \\ n_1\l_1=n_3\l_3  \\ P^+(n_1) = P^+(n_1 \l_1)=P^+(n_3 \l_3) = P^+(n_3)}} 1 = \sum_{\substack{n_1,n_3 \in \CS \\ n_1\l_1=n_3\l_3  \\ P^+(n_i) > |\l_i| }} 1 = \sum_{\substack{n_1,n_3 \in \CS \\ n_1\l_1=n_3\l_3}} 1 + O\bigg(\sum_{\substack{n_1,n_3 \leq N \\ n_1\l_1=n_3\l_3  \\ P^+(n_i) \leq L }} 1\bigg).
		\end{equation}
		By another application of Rankin's trick, one can compute that 
		\begin{equation}\label{eq: proof lem 4.3 1}
			\sum_{\substack{n_1,n_3 \le N \\ n_1\l_1=n_3\l_3  \\ P^+(n_i) \leq L }} 1 \ll \sum_{\substack{n_1\leq N \\ P^+(n_i) \leq L }} 1 \ll_A \frac{N}{(\log N)^A}.
		\end{equation}
		Since $n_i \geq 1$, the equation $n_1\l_1=n_3\l_3$ has no solutions if $\l_1\l_3 < 0$. Let us now separate the two cases for $z$. 
		
		If $z = \xi(N)$, then we are assuming $L \leq \xi(N) = z$. Thus, for any $n_1$ satisfying $P^-(n_1) > z$, we strictly have $P^-(n_1) > |\l_3|$. This guarantees that $(n_1,\l_3) = 1$. The equation $n_1\l_1=n_3\l_3$ then forces $n_1 = n_3,$ which implies $\l_1 = \l_3$. However, this contradicts our summation condition that $\l_1 \ne \l_3$. Therefore, when $z = \xi(N)$, the sum is completely empty, and $S_1 = 0$.
		
		On the other hand, if $z = 1$, the roughness condition is trivial ($\CS = [2,N]$). The equation $n_1\l_1 = n_3 \l_3$ is equivalent to $n_1 \l_1' = n_3 \l_3'$, where $\l_1' = \l_1/(\l_1,\l_3)$ and $\l_3' = \l_3/(\l_1,\l_3)$. Solutions to the equation with $n_1,n_3 \leq N$ are thus parameterized as $n_1 = k \l_3', n_3 = k\l_1'$, where $k \leq N/\max\{|\l_1'|,|\l_3'|\}.$ Thus, for $z=1$, we have:
		\begin{equation}\label{eq: proof lem 4.3 2}
			\sum_{\substack{n_1,n_3 \in \CS \\ n_1\l_1=n_3\l_3}} 1 = \bm{1}_{\l_1\l_3 > 0} \cdot \bigg\lfloor \frac{N}{\max\{|\l_1'|,|\l_3'|\}} \bigg\rfloor = \bm{1}_{\l_1\l_3 > 0}  \cdot \frac{N}{H(\l_1/\l_3)} + O(1).
		\end{equation}
		Combining \eqref{eq: S1 prime factor reduction}, \eqref{eq: proof lem 4.3 1} and \eqref{eq: proof lem 4.3 2} back into \eqref{eq: S1 halfway reduction}, we see that for $z=1$:
		\begin{align}\label{eq: S1 bound}
			S_1 = \bm{1}_{\substack{\l_1\l_2=\l_3\l_4 \\ \l_1 \ne \l_3, \l_2 \ne \l_4\\ \l_1\l_3,\l_2\l_4 > 0}} \times \frac{N^2}{H(\l_1/\l_3)H(\l_2/\l_4)} + O\bigg(\frac{N^2}{(\log N)^A}\bigg).
		\end{align}
		Substituting \eqref{eq: S1 bound}, \eqref{eq: S2 bound} and \eqref{eq:new S3 bound} into \eqref{eq: S1 S2 decomp} shows that for $z=1$,
		$$\sum_{\substack{n_1,\dots,n_4 \in \CS \\ n_1n_2=n_3n_4 \\ \{n_1\l_1,n_2\l_2\}=\{n_3\l_3,n_4\l_4\} \\ n_1 \ne n_3,n_2 \ne n_4 \\ P^+(n_1)=P^+(n_3) \\ P^+(n_2) = P^+(n_4)}} 1  = \bm{1}_{\substack{\l_1\l_2=\l_3\l_4 \\ \l_1 \ne \l_3, \l_2 \ne \l_4\\ \l_1\l_3,\l_2\l_4 > 0}} \times \frac{N^2}{H(\l_1/\l_3)H(\l_2/\l_4)} + O\bigg(\frac{N^2}{(\log N)^A} + N \cdot L^4\bigg).$$
		By choosing $A$ to be an appropriately large constant, and substituting this into \eqref{eq: verif cond 2 first}, we get that for $z=1$,
		\begin{equation}\label{eq: verif cond 2 before last}
			\sum_{\substack{n_1,\dots,n_4 \in \CS \\ n_1n_2=n_3n_4 \\ n_1 \ne n_3,n_2 \ne n_4 \\ P^+(n_1)=P^+(n_3) \\ P^+(n_2) = P^+(n_4)}} \hspace{-0.25cm} Q(\theta n_1)Q(\theta n_2)\overline{Q(\theta n_3)Q(\theta n_4)}= N^2 \hspace{-0.25cm} \sum_{\substack{\l_1,\dots,\l_4 \\ |\l_i| \leq L \\ \l_1\l_2=\l_3\l_4 \\ \l_1 \ne \l_3, \l_2 \ne \l_4\\ \l_1\l_3,\l_2\l_4 > 0}}\hspace{-0.25cm}\frac{ C(\bm{\l})}{H(\l_1/\l_3)H(\l_2/\l_4)} + O\bigg(\frac{N^2 L}{(\log N)^{\frac{4}{5}}}\bigg),
		\end{equation}
		since $N\cdot L^4$ is negligible compared to the other error terms. By using the definition of $\Delta$ from \eqref{eq: def Delta}, we see that the above equation simplifies precisely to $N^2 \Delta(Q_L) +O\big(N^2 L(\log N)^{-4/5}\big)$ when $z = 1$. Combining this with our earlier deduction that the sum is $O\big(N^2 L(\log N)^{-4/5}\big)$ when $z = \xi(N)$ completes the proof of the lemma.
	\end{proof}
	\begin{cor}\label{cor: sound-xu application}
		Let $Q_L$ be a trigonometric polynomial of the form \eqref{eq: def trig poly}, and let $z \in \{1, \xi(N)\}$. For any $\bm{u} = (u_1, u_2) \in \R^2$, define $u = u_1^2+u_2^2$. Then we have
		$$ \big|\phi(S_{Q_L}(N; z),\bm{u})- e^{-\frac{u}{4}}\big| \ll e^{\frac{u}{2}} \bigg( |\Delta(Q_L)|\bm{1}_{z=1} + \frac{L}{(\log N)^{4/5}} \bigg)^{1/2}. $$
	\end{cor}
	\begin{proof}
		We apply \Cref{thm:sound-xu 3.1} with $a_n = \bm{1}_{P^-(n) > z}(n)Q_L(\theta n)$. By \Cref{lem: variance estimate 1} and \Cref{lem: variance estimate 2}, the variance is $V = V_{Q_L}(N; z) \asymp N$. As noted at the beginning of this section, condition (1) of \Cref{thm:sound-xu 3.1} holds trivially. 
		
		For condition (2), \Cref{lem:verif cond 2 thm A} implies the required sum is bounded by $N^2( |\Delta(Q_L)|\bm{1}_{z=1} + O(L(\log N)^{-4/5}))$. To satisfy condition (2), which requires this sum to be $\le \epsilon^2 V^2$, we define our parameter $\epsilon$ as
		$$ \epsilon := C \bigg( |\Delta(Q_L)|\bm{1}_{z=1} + \frac{L}{(\log N)^{4/5}} \bigg)^{1/2} $$
		for a sufficiently large absolute constant $C>0$. 
		
		For condition (3), \Cref{lem: verif cond 3} bounds the corresponding sum by $O_A(N^2 (\log L)^4 (\log N)^{-A})$. By choosing $A$ sufficiently large (for instance, $A=2$), this bound is $\ll \epsilon^4 V^2$, meaning condition (3) is also satisfied. \Cref{thm:sound-xu 3.1} thus applies with $t^2 = (u_1^2+u_2^2)/2 = u/2$, yielding the claimed bound.
	\end{proof}
	We are now ready to prove the `if' direction of \Cref{thm:A}.
	\begin{proof}[Proof of `if' in \Cref{thm:A}]
		We are supposing $\Delta(g) = 0$, and wish to show that $S_{g}(N)$ converges in distribution to $\CC\CN(0,1)$ as $N \to \infty$. Write $\bm{u} = (u_1,u_2),$ and $u = u_1^2+u_2^2$. For a given random variable $X$, let $\phi(X,\bm{u})$ denote the characteristic function of $X$. That is: $$\phi(X,\bm{u}) := \E[e^{i u_1 \Re(X) + iu_2\Im(X)}].$$
		In our notation, a standard complex normal random variable has characteristic function $e^{-u/4}$. Thus, by L\'evy's continuity theorem, it suffices for us to prove that
		\begin{align}\lim_{N \to \infty} |\phi(S_g(N),\bm{u})- e^{-\frac{u}{4}}| = 0 \quad \forall \bm{u}\in \R^2 \label{char fn 1}.
		\end{align}
		Fix $L \geq 1$. By the triangle inequality, we have
		\begin{equation}\label{char fn 2}
			\big|\phi(S_g(N),\bm{u})-e^{-\frac{u}{4}}\big| \leq\big|\phi(S_g(N),\bm{u})- \phi(S_{Q_L}(N),\bm{u})\big| + \big|\phi(S_{Q_L}(N),\bm{u})- e^{-\frac{u}{4}}\big|.
		\end{equation}
		To bound the first term on the right hand side of \eqref{char fn 2}, we use the standard probabilistic inequality$$|\phi(X,\bm{u})-\phi(Y,\bm{u})| \leq \sqrt{u} \cdot \sqrt{\E |X-Y|^2}.$$
		This gives:
		\begin{equation*}
			\big|\phi(S_g(N),\bm{u})- \phi(S_{Q_L}(N),\bm{u})\big| \le \sqrt{u} \cdot \sqrt{\E |S_g(N)-S_{Q_L}(N)|^2}
		\end{equation*}
		We have
		\begin{align*}
			|S_g(N)-S_{Q_L}(N)| &= \frac{1}{\sqrt{V_{g}(N)}} \sum_{n \leq N} f(n) g(\theta n) -\frac{1}{\sqrt{V_{Q_L}(N)}} \sum_{n \leq N} f(n) Q_L(\theta n)\\
			&=  \frac{1}{\sqrt{V_{g}(N)}} \sum_{n \leq N} f(n) \Big(g(\theta n)-Q_L(\theta n)\Big) +\bigg(\frac{\sqrt{V_{Q_L}(N)}}{\sqrt{V_{g}(N)}} - 1\bigg)  S_{Q_L}(N)
		\end{align*}
		Using the inequality $|a+b|^2 \leq 2(|a|^2+|b|^2)$ and the orthogonality relations for Steinhaus RMFs, we thus have 
		\begin{align}\label{char fn 3}
			\big|\phi(S_g(N),\bm{u})- \phi(S_{Q_L}(N),\bm{u})\big| \le \sqrt{u\big(A_L(N) + B_L(N)\big)},
		\end{align}
		where \begin{align*}
			A_L(N) &:= \frac{2}{V_{g}(N)} \sum_{n \leq N}  \Big|g(\theta n)-Q_L(\theta n)\Big|^2,\\
			B_L(N) &:=  2\bigg(\frac{\sqrt{V_{Q_L}(N)}}{\sqrt{V_{g}(N)}} - 1\bigg)^2 \E \big|S_{Q_L}(N)\big|^2 = 2\bigg(\frac{\sqrt{V_{Q_L}(N)}}{\sqrt{V_{g}(N)}} - 1\bigg)^2.
		\end{align*}
		We now wish to bound the second term on the right hand side of \eqref{char fn 2}. By applying \Cref{cor: sound-xu application} with $z=1$, we immediately obtain:
		\begin{equation}\label{char fn 4}
			\big|\phi(S_{Q_L}(N),\bm{u})- e^{-\frac{u}{4}}\big| \ll e^{\frac{u}{2}} \cdot \sqrt{|\Delta(Q_L)| + \frac{L}{(\log N)^{4/5}}}.
		\end{equation}
		Substituting \eqref{char fn 3} and \eqref{char fn 4} into \eqref{char fn 2} gives:
		\begin{equation}\label{char fn 5}
			\big|\phi(S_g(N),\bm{u})-e^{-\frac{u}{4}}\big| \ll u^{1/2}\cdot\sqrt{A_L(N) + B_L(N)}+ e^{u/2} \cdot \sqrt{|\Delta(Q_L)| + L(\log N)^{-\frac{4}{5}}}.
		\end{equation}
		We have, by Weyl's equidistribution theorem, that
		$$\lim_{N \to \infty} A_L(N) = \lim_{N \to \infty} \sum_{n \leq N} \frac{2}{V_{g}(N)}   \Big|g(\theta n)-Q_L(\theta n)\Big|^2 = \frac{2}{\|g\|_2^2} \|g-Q_L\|_{2}^2,$$
		and $$\lim_{N \to \infty} B_L(N) = 2 \lim_{N \to \infty} \bigg(\frac{\sqrt{V_{Q_L}(N)}}{\sqrt{V_{g}(N)}} - 1\bigg)^2 =  2\bigg(\frac{\|Q_L\|_2}{\|g\|_2} - 1\bigg)^2.$$
		Thus, taking $N \to \infty$ in \eqref{char fn 5} gives
		\begin{equation}\label{char fn 6}
			\lim_{N \to \infty}\big|\phi(S_g(N),\bm{u})-e^{-\frac{u}{4}}\big| \ll u^{1/2} \cdot \sqrt{\frac{\|g-Q_L\|_{2}^2}{\|g\|_2^2}  +\bigg(\frac{\|Q_L\|_2}{\|g\|_2} - 1\bigg)^2}+ e^{\frac{u}{2}} \cdot \sqrt{|\Delta(Q_L)|}.
		\end{equation}
		Since $Q_L$ is defined as the truncated Fourier polynomial of $g$, classical Fourier theory implies that $\|g - Q_L\|_2 \to 0, \|Q_L\|_2 \to \|g\|$ as $L \to \infty$. Furthermore, it is plain to see that $\Delta(Q_L) \to \Delta(g) = 0$ as $L \to \infty$. Hence, taking $L \to \infty$ in the above inequality gives
		\begin{align*}
			\lim_{N \to \infty}\big|\phi(S_g(N),\bm{u})-e^{-\frac{u}{4}}\big| &\ll \lim_{L \to \infty} u^{1/2}\cdot \bigg[\sqrt{\frac{ \|g-Q_L\|_{2}^2 }{\|g\|_2^2}+ \bigg(\frac{\|Q_L\|_2}{\|g\|_2} - 1\bigg)^2  } + e^{\frac{u}{2}} \sqrt{|\Delta(Q_L)|}\bigg]\\ &= 0.
		\end{align*}
		This shows that \eqref{char fn 1} is true, and completes the proof.
	\end{proof}
	The proof of \Cref{thm:B}, given below, is largely similar to the above proof. As such, certain technical calculations are omitted. 
	\begin{proof}[Proof of \Cref{thm:B}]
		If $\xi(N) > \sqrt{N}$ and $N - \xi(N) \to +\infty$, then the convergence to a Gaussian limit follows trivially from the standard Central Limit Theorem. Thus, assume for the rest of the proof that $\xi(N) \le \sqrt{N}$. 
		
		We wish to show that $S_{g}(N;\xi(N))$ converges in distribution to $\CC\CN(0,1)$ as $N \to \infty$. To this end, we once again argue via characteristic functions and L\'evy's continuity theorem. By following the same argument as in the proof of `if' for \Cref{thm:A}, and applying \Cref{cor: sound-xu application} with $z=\xi(N)$ (which annihilates the $\Delta(Q_L)$ term), one can show that
		\begin{equation*}
			\big|\phi(S_{g}(N;\xi(N)),\bm{u})-e^{-\frac{u}{4}}\big| \ll u^{1/2} \cdot \sqrt{C_L(N) + D_L(N)}+ e^{u/2}\cdot \sqrt{\frac{L}{(\log N)^{4/5}}},
		\end{equation*}
		where
		\begin{align*}
			C_L(N) &:= \frac{2}{V_{g}(N;\xi(N))} \sum_{\substack{n \leq N \\ P^-(n) > \xi(N)}}  \Big|g(\theta n)-Q_L(\theta n)\Big|^2,\\
			D_L(N) &:=  2\bigg(\frac{\sqrt{V_{Q_L}(N;\xi(N))}}{\sqrt{V_{g}(N;\xi(N))}} - 1\bigg)^2 \E \big|S_{Q_L}(N;\xi(N))\big|^2 = 2\bigg(\frac{\sqrt{V_{Q_L}(N;\xi(N))}}{\sqrt{V_{g}(N;\xi(N))}} - 1\bigg)^2.
		\end{align*}
		By \Cref{lem: variance estimate 2}, we have
		$$\lim_{N \to \infty} C_L(N) = \frac{2}{\|g\|_2^2} \|g-Q_L\|_{2}^2,$$
		and $$\lim_{N \to \infty} D_L(N) =  2\bigg(\frac{\|Q_L\|_2}{\|g\|_2} - 1\bigg)^2.$$
		Taking $N \to \infty$ and then $L \to \infty$ thus implies
		\begin{align*}
			\lim_{N\to \infty}\big|\phi(S_{g}(N;\xi(N)),\bm{u})-e^{-\frac{u}{4}}\big| &\ll \lim_{L \to \infty} u^{1/2}\sqrt{\frac{\|g-Q_L\|_{2}^2}{\|g\|_2^2}  +\bigg(\frac{\|Q_L\|_2}{\|g\|_2} - 1\bigg)^2} \\ &= 0.
		\end{align*}
		This proves that $S_{g}(N;\xi(N))$ converges in distribution to a standard complex normal random variable.
		
	\end{proof}

	\section{Proof of `only if' in \Cref{thm:A}}\label{sec: proof only if}
	The following Proposition is arguably the most important result in the present paper, as it establishes a new technique for bounding moments of sums of RMFs, so long as they are weighted by sequences that exhibit cancellation on the primes. In our case, the sequence in question is exponentials $e(\theta n)$. It would be interesting to try to extend this proposition to more general types of weights.
	\begin{prop}\label{thm: bounded moments}
		Let $Q_L$ be defined as in \eqref{eq: def trig poly}. Let $\theta \in \R\setminus\Q$ satisfy the Diophantine condition \eqref{eq:sound-xu badness approx}. Then, for each $k \in \N$, we have uniformly for $L \le (\log N)^{1/10}$ that
		$$\E\Big[\big|S_{Q_L}(N)\big|^{2k}\Big] \ll_k 1 +\frac{\sqrt{L}(\log L)^{2k}}{(\log N)^{\frac{115 - 14k^2 + 4k}{12} - o(1)}}.$$
		In particular, we have uniformly for $L \le (\log N)^{1/10}$ that 
		$$\E\big[\big|S_{Q_L}(N)\big|^{6}\big] \ll 1  + \frac{\sqrt{L}(\log L)^6 }{(\log N)^{1/13}}.$$
	\end{prop}
	\begin{proof}
		Recall the martingale difference sequence setup from \Cref{sec:MDS}. For our truncated sum, we define the weights $a_n = Q_L(\theta n)$ and consider the martingale differences
		$$\CM_p(N) = \sum_{\substack{n \leq N \\ P^+(n) = p}} a_n f(n) = f(p) \sum_{\substack{m \leq N/p \\ P^+(m) \leq p}} a_{mp} f(m).$$
		Since $|f(p)|^2 = 1$, the absolute square of each martingale difference isolates the prime $p$ from the remaining smooth factors:
		$$|\CM_p(N)|^2 = \sum_{\substack{m,n \leq N/p \\ P^+(mn) \leq p}} a_{mp} \overline{a_{np}} f(m) \overline{f(n)}.$$
		By applying the Burkholder--Davis--Gundy inequality (\Cref{thm:BDG}) to the martingale $$S_{Q_L}(N)\cdot\sqrt{V_{Q_L}(N)} = \sum_{p \leq N} \CM_p(N),$$ we bound the $2k$-th moment of our sum by the $k$-th moment of its quadratic variation:
		\begin{equation*}
			\E\Big[\big|S_{Q_L}(N)\big|^{2k}\Big] \ll_k \frac{1}{V_{Q_L}(N)^k} \E\bigg[ \bigg( \sum_{p \leq N} |\CM_p(N)|^2 \bigg)^k \bigg].
		\end{equation*}
		Expanding the $k$-th power of the sum over primes $p_1, \dots, p_k$ gives a product of $k$ independent terms $|\CM_{p_j}(N)|^2$. Taking the expectation over the independent Steinhaus variables $f$ yields the condition that $\E[f(m_1\cdots m_k)\overline{f(n_1\cdots n_k)}] = 1$ if $m_1\cdots m_k = n_1\cdots n_k$, and $0$ otherwise. By substituting $a_n = Q_L(\theta n) = \sum_{|r| \le L} c_r e(\theta n r)$, this expectation isolates the diagonal terms where the products exactly match, yielding:
		\begin{equation}\label{eq: BDG 1}
			\E\Big[\big|S_{Q_L}(N)\big|^{2k}\Big] \ll_k  \frac{1}{V_{Q_L}(N)^k}\sum_{p_1,\dots,p_k \leq N}\sum_{\substack{r_1,\dots ,r_k \\ s_1,\dots ,s_k \\ |r_i|,|s_i| \leq L}}C(\bm{r,s})\sum_{\substack{m_1,\dots, m_k\\ n_1,\dots ,n_k\\ m_i,n_i \leq N/p_i\\ P^+(m_in_i) \leq p_i\\ m_1\cdots m_k=n_1 \cdots n_k}} \prod_{j=1}^k e\big[\theta p_j(m_jr_j-n_js_j)\big],
		\end{equation}
		where $C(\bm{r,s}) := \prod_{i = 1}^k c_{r_i}\overline{c_{s_i}}.$
		Here we distinguish two cases, depending on whether or not all the phases $m_jr_j-n_js_j$ are zero. 
		
		\textbf{Case 1}: $m_jr_j-n_js_j = 0 $ \textit{for all} $j$. In this case, we have 
		$$m_1 \cdots m_k r_1 \cdots r_k = n_1 \cdots n_k s_1 \cdots s_k.$$
		Since $\prod_{i = 1}^k m _i = \prod_{i = 1}^k n_i$ in \eqref{eq: BDG 1}, then we also have $\prod_{i = 1}^k r_i = \prod_{i = 1}^k s_i$. The contribution of such terms to \eqref{eq: BDG 1} is
		\begin{equation}\label{eq: BDG 2}
			\frac{1}{V_{Q_L}(N)^k}\sum_{p_1,\dots ,p_k \leq N}\sum_{\substack{r_1,\dots ,r_k \\ s_1,\dots ,s_k \\ |r_i|,|s_i| \leq L \\ r_1 \cdots r_k = s_1 \cdots s_k}}C(\bm{r,s})\sum_{\substack{m_1,\dots, m_k\\ n_1,\dots ,n_k\\ m_i,n_i \leq N/p_i\\ P^+(m_in_i) \leq p_i\\ m_ir_i=n_is_i}} 1.
		\end{equation}
		Note that 
		$$\sum_{\substack{m,n \leq N/p\\ P^+(mn) \leq p\\ mr=ns}} 1 = \Psi\bigg(\frac{N}{p H(r/s)},p\bigg),$$
		where $\Psi(x,y)$ counts the number of $y-$smooth numbers not exceeding $x$, and $H$ denotes the na\"ive (Weil) height. Thus, \eqref{eq: BDG 2} becomes
		\begin{equation}\label{eq: BDG 3}
			\frac{1}{V_{Q_L}(N)^k}\sum_{\substack{r_1,\dots ,r_k \\ s_1,\dots ,s_k \\ |r_i|,|s_i| \leq L}}C(\bm{r,s}) \prod_{j = 1}^k  \bigg(\sum_{p_j \leq N}\Psi\Big(\frac{N}{p_j H(r_j/s_j)},p_j\Big)\bigg).
		\end{equation}
		
		Since $\sum_{p \leq N}\Psi(X/p,p) = \Psi(X,N) -1 \le X,$ then \eqref{eq: BDG 3} can be bounded as
		
		\begin{equation}\label{eq: BDG 4}
			\ll_k \frac{ N^k}{V_{Q_L}(N)^k}\sum_{\substack{r_1,\dots ,r_k \\ s_1,\dots ,s_k \\ |r_i|,|s_i| \leq L \\ r_1 \cdots r_k = s_1 \cdots s_k}}\frac{|C(\bm{r,s})|}{\prod_{j = 1}^k H(r_j/s_j)} \ll_k \frac{N^k}{V_{Q_L}(N)^k}\sum_{\substack{r_1,\dots ,r_k \\ |r_i|\leq L}}\frac{\tau_k(|r_1\cdots r_k|)}{|r_1\cdots r_k|^2},
		\end{equation}
		since $|C(\bm{r,s})| \ll |r_1 \dots  r_k s_1 \dots  s_k|^{-1},$ and $r_1 \cdots r_k = s_1\cdots s_k.$ By \Cref{lem: variance estimate 1} and the sub-multiplicativity of $\tau_k$, we see that the above can be bounded as
		\begin{equation}\label{eq: BDG 5}
			\frac{N^k}{V_{Q_L}(N)^k}\sum_{\substack{r_1,\dots ,r_k \\ |r_i|\leq L}}\frac{\tau_k(|r_1\cdots r_k|)}{|r_1\cdots r_k|^2} \ll_k \bigg(\sum_{\substack{r \ge 1}}\frac{\tau_k(r)^2}{r^2}\bigg)^k \ll_k 1.
		\end{equation}
		This shows that the terms in \eqref{eq: BDG 1} for which $m_jr_j-n_js_j = 0 $ for all $j$ contribute an amount which is at most a constant depending on $k$.
		
		\textbf{Case 2}: There exists $j$ such that $m_jr_j-n_js_j \ne 0 $. In this case, suppose without loss of generality that $m_1r_1-n_1s_1 \ne 0.$ By interchanging orders of summation, we rewrite the contribution to the sum on the right hand side of \eqref{eq: BDG 1} as
		\begin{equation}\label{eq: BDG 6}
			\sum_{\substack{r_1,\dots ,r_k \\ s_1,\dots ,s_k \\ |r_i|,|s_i| \leq L}}C(\bm{r,s})\sum_{\substack{m_1,\dots, m_k\leq N\\ n_1,\dots ,n_k \leq N\\ m_1\cdots m_k=n_1 \cdots n_k \\ m_1r_1 \ne n_1s_1}} \prod_{j=1}^k	\bigg|\sum_{\substack{ P^+(m_jn_j) \leq p_j \leq \frac{N}{\max(n_j,m_j)}}} e\big[\theta p_j(m_jr_j-n_js_j)\big]\bigg|.
		\end{equation}
		
		By taking absolute values, and bounding all exponential prime sums except for the first one by $N/\max(n_i,m_i)$, we see that \eqref{eq: BDG 6} is 
		\begin{equation}\label{eq: BDG 7}
			\ll N^{k-1}\sum_{\substack{r_1,\dots ,r_k \\ s_1,\dots ,s_k \\ |r_i|,|s_i| \leq L}}|C(\bm{r,s})|\sum_{\substack{m_1,\dots, m_k\leq N\\ n_1,\dots ,n_k \leq N\\ m_1\cdots m_k=n_1 \cdots n_k}} \prod_{j = 2}^{k} \bigg(\frac{1}{\max(n_j,m_j)}\bigg) \bigg| \sum_{\substack{ P^+(m_1n_1) \leq p_1 \leq \frac{N}{\max(n_1,m_1)}}} e\big[\theta p_1(m_1r_1-n_1s_1)\big]\bigg|.
		\end{equation}
		Suppose that $\max(n_1,m_1) > N \exp(-\sqrt{\log N}).$ Let $W = N \exp(-\sqrt{\log N})$. By symmetry, we may assume $m_1 > W$. The constraint $P^+(m_1n_1) \le p_1 \le N/\max(n_1,m_1)$ forces $P^+(m_1) \le N/W = \exp(\sqrt{\log N})$. Bounding $1/\max(n_j, m_j) \le 1/m_j$, the contribution of this tail to \eqref{eq: BDG 7} is bounded by:
		\begin{align*}
			&\ll N^{k}\sum_{\substack{r_1,\dots ,r_k \\ s_1,\dots ,s_k \\ |r_i|,|s_i| \leq L}}|C(\bm{r,s})|\sum_{\substack{m_1,\dots, m_k\leq N\\ n_1,\dots ,n_k \leq N\\ m_1\cdots m_k=n_1 \cdots n_k}} \prod_{j = 1}^{k} \bigg(\frac{1}{m_j}\bigg) \bm{1}_{\substack{m_1 > W \\ P^+(m_1) \le \exp(\sqrt{\log N})}}\\
			&\ll N^{k}\sum_{\substack{r_1,\dots ,r_k \\ s_1,\dots ,s_k \\ |r_i|,|s_i| \leq L}}|C(\bm{r,s})|\sum_{\substack{ m_1,\dots ,m_k \leq N\\ m_1 > W \\ P^+(m_1) \le \exp(\sqrt{\log N})}} \frac{\tau_k(m_1\cdots m_k)}{m_1\cdots m_k}\\
			&\ll_k N^{k}\sum_{\substack{r_1,\dots ,r_k \\ s_1,\dots ,s_k \\ |r_i|,|s_i| \leq L}}|C(\bm{r,s})| \bigg(\sum_{m \leq N} \frac{\tau_k(m)}{m}\bigg)^{k-1}\bigg(\sum_{\substack{ m_1 > W \\ P^+(m_1) \leq \exp(\sqrt{\log N})}} \frac{\tau_k(m_1)}{m_1}\bigg).
		\end{align*}
		
		To bound the final sum over $m_1$, we apply Rankin's trick with $\delta = 1/\sqrt{\log N}$. By multiplying the summands by $(m_1/W)^\delta \ge 1$, we obtain:
		\begin{align*}
			\sum_{\substack{ m_1 > W \\ P^+(m_1) \leq \exp(\sqrt{\log N})}} \frac{\tau_k(m_1)}{m_1} &\le W^{-\delta} \sum_{P^+(m) \le \exp(\sqrt{\log N})} \frac{\tau_k(m)}{m^{1-\delta}} \\
			&\le \exp(-\sqrt{\log N}) \prod_{p \le \exp(\sqrt{\log N})} \bigg(1 - \frac{1}{p^{1-\delta}}\bigg)^{-k} \\
			&\ll_k \exp(-\sqrt{\log N}) (\log N)^{k/2}.
		\end{align*}
		
		On the other hand, the remaining $k-1$ sums of $\tau_k(m)/m$ contribute $(\log N)^{k(k-1)}$. Hence, the total bound for the tail is $O_k(N^k \exp(-\frac{1}{2}\sqrt{\log N}))$, which decays sub-exponentially and is thus negligible.
		It remains to bound \eqref{eq: BDG 6} in the case where $\max\{n_1,m_1\} \leq W$. By symmetry, we may assume without loss of generality that $m_1 \ge n_1$. So we must bound
		\begin{equation}\label{eq: BDG 7.5}
			N^{k-1}\sum_{\substack{r_1,\dots ,r_k \\ s_1,\dots ,s_k \\ |r_i|,|s_i| \leq L}}|C(\bm{r,s})|\sum_{\substack{m_1,\dots, m_k\leq N\\ n_1,\dots ,n_k \leq N\\ m_1\cdots m_k=n_1 \cdots n_k \\ n_1 \le m_1 \leq W}} \prod_{j = 2}^{k} \bigg(\frac{1}{\max(n_j,m_j)}\bigg) \bigg| \sum_{\substack{ P^+(m_1n_1) \leq p_1 \leq \frac{N}{m_1}}} e\big[\theta p_1(m_1r_1-n_1s_1)\big]\bigg|.
		\end{equation}
		
		Let $\gamma$ be a parameter to be optimized later. Let the set $\CM$ be defined as
		$$\CM =\CM(r_1,s_1,n_1) = \bigg\{m_1 \leq W : \exists q \in \Z_{\ne 0}, |q| \leq (\log N)^{\gamma} \text{ s.t. } \|\theta q (m_1r_1-n_1s_1)\| \leq \frac{W(\log N)^{\gamma}}{N} \bigg\}.$$
		
		In the case when $m_1 \in \CM$, we bound the prime sum in \eqref{eq: BDG 7.5} trivially by $N/m_1$.
		
		On the other hand, in the case when $m_1 \not\in \CM$, we use Dirichlet's approximation theorem to find a rational number $a/q$ with $(a,q) = 1$ and $1 \leq q \leq N/(m_1(\log N)^{\gamma})$ such that $|\theta (m_1r_1-n_1s_1) - a/q| \leq m_1(\log N)^{\gamma}/(q N) \le W(\log N)^\gamma/(q N)$. Since $m_1 \not \in \CM$, it must be the case that $q > (\log N)^{\gamma}$. Given this rational approximation with large enough $q$, we apply \Cref{thm: exp sum primes}. Because $N/m_1 \ge \exp(\sqrt{\log N})$, the term $\exp(-\frac{1}{2}\sqrt{\log (N/m_1)})$ coming out of \Cref{thm: exp sum primes} is negligible, and we can omit it. We find:
		\begin{align*}
			\bigg| \bm{1}_{m_1 \not \in \CM} \sum_{\substack{ P^+(m_1n_1) \leq p_1 \leq \frac{N}{m_1}}} e\big[\theta p_1(m_1r_1-n_1s_1)\big]\bigg| &\ll \frac{N}{m_1}(\log N)^{3/4+o(1)} \bigg( \frac{1}{q^{1/2}} +\frac{(m_1 q \log q )^{1/2}}{N^{1/2}} \bigg)\\
			&\ll \frac{N}{m_1}(\log N)^{3/4+o(1)} \bigg( (\log N)^{-\gamma/2} +\frac{(\log N )^{1/2}}{(\log N)^{\gamma/2}} \bigg)\\
			&\ll \frac{N}{m_1 (\log N)^{\gamma/2 - 5/4 -o(1)}}.
		\end{align*}
		
		We have bounded the prime sum in \eqref{eq: BDG 7.5} in the two cases $m_1 \in \CM$ and $m_1 \not \in \CM$. By combining these cases, we see that \eqref{eq: BDG 7.5} is bounded by:
		\begin{equation}\label{eq: BDG 8}
			\ll  N^{k}\sum_{\substack{r, s \\ |r_i|,|s_i| \leq L}}|C(\bm{r,s})|\bigg(\sum_{\substack{m_i, n_i \leq N \\  m_1\cdots m_k= n_1 \cdots n_k}} \frac{\bm{1}_{m_1 \in \CM} + \frac{\bm{1}_{m_1 \not\in \CM}}{(\log N)^{\gamma/2-5/4-o(1)}} }{\max(n_1,m_1) \cdots \max(n_k,m_k)} \bigg).
		\end{equation}
		We have, by sub-multiplicativity of $\tau_k$ and the fact that $\sum_{n \leq t}\tau_k(n)/n \ll (\log t)^k$, that
		\begin{align}\label{eq: BDG 12}
			\sum_{\substack{m, n \leq N\\ m_1\cdots m_k= n_1 \cdots n_k}}  \frac{\bm{1}_{m_1 \not\in \CM}}{\max(n_1,m_1) \cdots \max(n_k,m_k)} &\ll  \sum_{\substack{m_1,\dots, m_k\leq N}}  \frac{\tau_k(m_1 \cdots m_k)}{m_1 \cdots m_k} \ll (\log N)^{k^2}.
		\end{align}
		On the other hand, the sum over $m_1 \in \CM$ can be bounded as:
		\begin{align}
			\sum_{\substack{m, n \leq N\\ \prod m_i=\prod n_i}}  \frac{\bm{1}_{m_1 \in \CM}}{\max(n_1,m_1) \cdots \max(n_k,m_k)} &\ll \bigg(\sum_{\substack{m_1 \leq N\\ m_1 \in \CM}} \frac{\tau_k(m_1)}{m_1}\bigg)\bigg( \sum_{\substack{m \leq N}}  \frac{\tau_k(m)}{m}\bigg)^{k-1}. \label{eq: BDG 10}\\
			& \nonumber \ll   \bigg(\sum_{\substack{m_1 \leq N\\ m_1 \in \CM}} \frac{\tau_k(m_1)}{m_1}\bigg)(\log N)^{k^2-k}. 
		\end{align}
		Applying Lemma \ref{lem: badness spacing} with $X=W, Q=(\log N)^\gamma, s=r_1, y=n_1s_1$, and constant $\eta = W(\log N)^\gamma/N$, we see that any distinct $m_1, m_1' \in \CM$ are separated by at least $\delta$, where
		$$ \delta = \frac{1}{|r_1| (\log N)^{2\gamma}} \bigg( \log \frac{C N}{2 W (\log N)^\gamma} \bigg)^{127} \gg \frac{(\log N)^{127/2 - 2\gamma}}{L}, $$
		having used the fact that $W = N \exp(-\sqrt{\log N})$. 
		
		Furthermore, since we assumed $m_1r_1 - n_1s_1 \ne 0$, the Diophantine condition \eqref{eq:sound-xu badness approx} dictates that elements in $\CM$ cannot be arbitrarily small. Specifically, for any $m_1 \in \CM$, there exists $q \ne 0$ with $|q| \le (\log N)^\gamma$ such that:
		$$ C \exp\big(-|q(m_1r_1 - n_1s_1)|^{1/127}\big) \le \|\theta q (m_1r_1-n_1s_1)\| \le \frac{W(\log N)^\gamma}{N} = (\log N)^\gamma \exp(-\sqrt{\log N}). $$
		Taking logarithms and rearranging yields $|q(m_1r_1 - n_1s_1)| \ge (\log N)^{127/2}/2^{127}$ for large $N$. Since $|q| \le (\log N)^\gamma$ and $|m_1r_1 - n_1s_1| \le 2L \max(n_1, m_1) = 2L m_1$, we divide to find:
		$$ m_1 \ge Y := \frac{(\log N)^{127/2 - \gamma}}{2^{128} L} \gg (\log N)^{127/2 - \gamma - 1/9}. $$
		
		Applying Corollary \ref{cor: log weights} with these parameters yields:
		$$ \sum_{\substack{m_1 \leq W \\ m_1 \in \CM}} \frac{\tau_k(m_1)}{m_1} \ll_k \frac{\sqrt{L}}{(\log N)^{\frac{127}{4} - \gamma}} (\log N)^{\frac{k^2+1}{2}} + Y^{-1/2} (\log N)^{\frac{k^2-1}{2}}. $$
		Because $Y$ is a large power of $\log N$, the $Y^{-1/2}$ term decays faster than the first term and is therefore absorbed, yielding:
		\begin{equation}\label{eq: BDG 10.5}
			\sum_{\substack{m_1 \leq W \\ m_1 \in \CM}} \frac{\tau_k(m_1)}{m_1} \ll_k \frac{\sqrt{L}}{(\log N)^{\frac{127}{4} - \gamma - \frac{k^2+1}{2}}}.
		\end{equation}
		
		Substituting \eqref{eq: BDG 10.5} into \eqref{eq: BDG 10}, we see that 
		\begin{equation}\label{eq: BDG 11}
			\sum_{\substack{m,n \leq N\\ \prod m_i=\prod n_i}}  \frac{\bm{1}_{m_1 \in \CM}}{\max(n_1,m_1) \cdots \max(n_k,m_k)} \ll\frac{\sqrt{L}}{(\log N)^{\frac{127}{4} - \gamma - \frac{3}{2}k^2 +k -\frac{1}{2} -o(1)}}.
		\end{equation}
		
		By combining \eqref{eq: BDG 12} and \eqref{eq: BDG 11} back into \eqref{eq: BDG 8}, the contribution of terms in Case 2 to \eqref{eq: BDG 1} is 
		\begin{equation}\label{eq: BDG 13}
			\ll_k \frac{N^k}{V_{Q_L}(N)^k} \cdot \frac{\sqrt{L}}{(\log N)^{\min\{\gamma/2-5/4 -k^2, \; \frac{127}{4} - \gamma-\frac{3}{2}k^2+k-\frac{1}{2}\}-o(1)}}\sum_{\substack{\bm{r},\bm{s} \\ |r_i|,|s_i| \leq L}}|C(\bm{r,s})|. 
		\end{equation}
		
		By \Cref{lem: variance estimate 1}, and the fact that $|C(\bm{r,s})| \ll |r_1\cdots r_k s_1\cdots s_k|^{-1},$ we see that \eqref{eq: BDG 13} is bounded by
		$$\ll_k \frac{\sqrt{L}(\log L)^{2k}}{(\log N)^{\min\{\gamma/2-5/4 -k^2, \; \frac{127}{4} - \gamma-\frac{3}{2}k^2+k-\frac{1}{2}\}-o(1)}}.$$
		
		We choose $\gamma = \frac{65}{3} - \frac{k^2}{3}+\frac{2k}{3}$ so as to maximize the above minimum. The bound becomes
		$$\ll_k \frac{\sqrt{L}(\log L)^{2k}}{(\log N)^{\frac{115 - 14k^2 + 4k}{12} - o(1)}}.$$
		This completes the estimate for Case 2.
		
		By combining the estimates from both cases, we conclude the general statement of the theorem. To prove the claim following the general statement, simply apply the Theorem with $ k = 3$. 
	\end{proof}
	We are nearly ready to prove `only if' in \Cref{thm:A}. The last remaining piece is an asymptotic calculation of the fourth moment of $S_{Q_L}(N)$.
	\begin{lem}\label{lem:fourth moment}
		Let $Q_L$ be defined as in \eqref{eq: def trig poly}. Let $\theta$ be defined as in \eqref{eq:sound-xu badness approx}. Then, uniformly for all $L \le (\log N)^{1/10}$, one has
		$$\E[|S_{Q_L}(N)|^4] = 2+\frac{2 N^2 \Delta(Q_L)}{V_{Q_L}(N)^2} + O\bigg( \frac{ L (\log L)^4}{(\log N)^{\frac{4}{5}}}  \bigg).$$
	\end{lem}
	\begin{proof}
		First, we expand the fourth moment and separate it into two parts: \begin{equation}\label{4th moment expansion s1 s2}
			\E[|S_{Q_L}(N)|^4] = \frac{1}{V_{Q_L}(N)^2}\sum_{\substack{n_1,\dots n_4 \leq N \\ n_1n_2=n_3n_4}}Q(\theta n_1)Q(\theta n_2)\overline{Q(\theta n_3)Q(\theta n_4)} = \frac{1}{V_{Q_L}(N)^2}(S_1+S_2),
		\end{equation}
		where
		\begin{align*}
			S_1 &:= \sum_{\substack{n_1,\dots ,n_4 \leq N \\ \{n_1,n_2\}=\{n_3,n_4\}}}Q(\theta n_1)Q(\theta n_2)\overline{Q(\theta n_3)Q(\theta n_4)}\\
			S_2 &:= \sum_{\substack{n_1,\dots ,n_4 \leq N \\ n_1n_2=n_3n_4 \\ \{n_1,n_2\}\ne\{n_3,n_4\}}}Q(\theta n_1)Q(\theta n_2)\overline{Q(\theta n_3)Q(\theta n_4)}
		\end{align*}
		Let us simplify $S_1$. We can rewrite the condition $\{n_1,n_2\}=\{n_3,n_4\}$ as $n_1=n_3,n_2=n_4$ or $n_1=n_4,n_2=n_3$. In either case, the factor $Q(\theta n_1)Q(\theta n_2)\overline{Q(\theta n_3)Q(\theta n_4)}$ becomes $|Q(\theta n_1)Q(\theta n_2)|^2$. We can thus rewrite $S_1$ as a part where $n_1=n_3,n_2=n_4$ plus a part where $n_1=n_4,n_2=n_3$, and minus the part where $n_1=n_2=n_3=n_4$ which will otherwise be counted twice:
		\begin{align*}
			S_1 &= 2\sum_{\substack{n_1,n_2 \leq N}}|Q(\theta n_1)Q(\theta n_2)|^2 - \sum_{n \leq N}|Q(\theta n)|^4 \\ 
			&= 2\bigg(\sum_{\substack{n \leq N}}|Q(\theta n)|^2\bigg)^2 - \sum_{n \leq N}|Q(\theta n)|^4
		\end{align*}
		To bound the second sum, we expand $|Q_L(\theta n)|^4$ as a trigonometric polynomial of degree at most $4L$. Because $\hat{g}(k) \ll 1/|k|$, the sum of all coefficients of this polynomial is bounded by $O((\log L)^4)$. By applying our Diophantine condition \eqref{eq:sound-xu badness approx}, the sum over non-zero frequencies is bounded by $O\big(L (\log L)^4 \exp((4L)^{1/127})\big)$. Because we assume $L \le (\log N)^{1/10}$, this off-diagonal error is bounded by $N^{o(1)}$, which is $o(N)$. The main diagonal term is $N \int |Q_L|^4 \ll N (\log L)^2$. Thus, the second sum is bounded by $O(N(\log L)^2)$. We thus have
		$$S_1 = 2V_{Q_L}(N)^2 +O(N (\log L)^2).$$
		
		We must now deal with the sum $S_2$. By using the definition of $Q_L$, we can write
		$$S_2 = \sum_{\substack{\l_1,\dots ,\l_4 \\ |\l_i| \leq L}}C(\bm{\l})\sum_{\substack{n_1,\dots n_4 \leq N \\ n_1n_2=n_3n_4 \\ \{n_1,n_2\}\ne\{n_3,n_4\}}}e\big[\theta (n_1\l_1+n_2\l_2-n_3\l_3-n_4\l_4)\big].$$
		
		We partition $S_2$ into two contributions: an aligned contribution where the frequencies perfectly cancel the phases (i.e., $\{n_1\l_1,n_2\l_2\}=\{n_3\l_3,n_4\l_4\}$), and the complementary non-aligned contribution where $\{n_1\l_1,n_2\l_2\}\ne\{n_3\l_3,n_4\l_4\}$. 
		
		For the aligned contribution, the set equality holds if and only if either (I) $n_1\l_1=n_3\l_3$ and $n_2\l_2=n_4\l_4$, or (II) $n_1\l_1=n_4\l_4$ and $n_2\l_2=n_3\l_3$. We analyze case (I); case (II) is symmetrically identical. In case (I), combining the phase alignment with the multiplicative constraint yields 
		$$ (n_1 n_2) \l_1 \l_2 = (n_3 n_4) \l_3 \l_4. $$
		Since $n_1n_2 = n_3n_4 \ge 1$, this forces $\l_1\l_2 = \l_3\l_4$. Furthermore, the overarching sum constraint $\{n_1, n_2\} \ne \{n_3, n_4\}$ dictates that we do not match on the straight diagonal ($n_1=n_3$ and $n_2=n_4$). If $n_1 = n_3$, then the relation $n_1\l_1 = n_3\l_3$ immediately implies $\l_1 = \l_3$. Thus, to exclude the straight diagonal, we enforce $\l_1 \ne \l_3$ and $\l_2 \ne \l_4$. 
		
		The intersection of case (I) and case (II) occurs when $n_1=n_4$ and $n_2=n_3$, which contributes a negligible $O(N(\log L)^4)$ overlapping terms. Stripping this overlap, the number of valid integer tuples in $[1, N]^4$ satisfying the conditions of case (I) relies only on the linear equations $n_1\l_1=n_3\l_3$ and $n_2\l_2=n_4\l_4$. As derived earlier in \eqref{eq: proof lem 4.3 2} (where the $P^+$ sieve conditions are notably absent), the number of such solutions is precisely 
		$$ \sum_{\substack{n_1,\dots,n_4 \leq N \\ n_1\l_1 = n_3 \l_3 \\ n_2\l_2 = n_4\l_4}} 1 = \bm{1}_{\l_1\l_3 > 0, \l_2\l_4>0} \frac{N^2}{H(\l_1/\l_3)H(\l_2/\l_4)} + O(N). $$ 
		Summing this over the allowed frequencies $\bm{\l}$ exactly recovers $N^2 \Delta(Q_L) + O(N(\log L)^4)$. Adding the identical contribution from case (II), the total phase-aligned contribution to $S_2$ is exactly $2N^2 \Delta(Q_L) + O(N(\log L)^4)$.
		
		For the non-aligned contribution, we use the parameterization of \Cref{lem:parameterization} to write $n_1 = ga$, $n_2 = hb$, $n_3 = gb$, and $n_4 = ha$, where $(a,b)=1$ and $a \ne b$. This transforms the non-aligned sum into:
		\begin{equation}\label{eq: non-aligned T3 sum}
			\sum_{\substack{\l_1,\dots,\l_4 \\ |\l_i| \leq L}} C(\bm{\l}) \sum_{\substack{\max(a,b) \times \max(g,h) \leq N \\ a \ne b, \; (a,b) = 1 \\ \{ga\l_1,hb\l_2\} \ne \{gb\l_3,ha\l_4\}}} e\big[\theta (ga\l_1+hb\l_2-gb\l_3-ha\l_4)\big].
		\end{equation}
		This sum is structurally identical to the one bounded in the proof of \Cref{lem:reduction cond 2}, with the only difference being that the conditions $P^+(ab) \leq \min\{P^+(g),P^+(h)\}$ are absent. Omitting these sieve constraints only simplifies the argument, and by repeating the same proof as in that lemma, we find that the non-aligned case contributes an error of at most $O\big(N^2 L(\log L)^4 (\log N)^{-4/5}\big)$.
		
		Combining our estimates for the aligned and non-aligned portions of $S_2$ back into \eqref{4th moment expansion s1 s2}, and using the fact that $V_{Q_L}(N) \asymp N$, we conclude that
		$$\E[|S_{Q_L}(N)|^4] = 2+\frac{2 N^2 \Delta(Q_L)}{V_{Q_L}(N)^2} + O\bigg(\frac{(\log L)^4}{N} + \frac{ L (\log L)^4}{(\log N)^{\frac{4}{5}}}  \bigg).$$
		We note that the second error term is asymptotically larger than the first, so the $1/N$ error term may be ignored. This concludes the proof of the lemma.
	\end{proof}
	
	\begin{proof}[Proof of `only if' in \Cref{thm:A}]
		Assume for the sake of contradiction that $S_{g}(N)$ converges in distribution to a standard complex normal random variable $Z \sim \CC\CN(0,1)$ as $N \to \infty$, but that $\Delta(g) \ne 0$. 
		
		Let $L_N$ be a sequence growing to infinity sufficiently slowly such that $L_N \le (\log N)^{1/100}$ and $\E[|S_g(N) - S_{Q_{L_N}}(N)|^2] \to 0$ as $N \to \infty$. The existence of such a slowly growing sequence is guaranteed since $\|g - Q_L\|_2 \to 0$ as $L \to \infty$. Let $X_N := S_{Q_{L_N}}(N)$. Because $S_g(N) \to^d Z$ and the $L^2$-distance between $S_g(N)$ and $X_N$ vanishes, Slutsky's theorem implies that we also have $X_N \xrightarrow{d} Z$.
		
		By \Cref{thm: bounded moments}, since $L_N \le (\log N)^{1/100}$, the sixth moment of $X_N$ is uniformly bounded:
		\begin{equation}\label{eq: bound 6}
			\sup_{N \geq 1} \E\big[|X_N|^6\big] \le C,
		\end{equation}
		for some constant $C > 0$. We now compute the limit of the fourth moment $\E[|X_N|^4]$ in two different ways to extract a contradiction.
		
		On one hand, by \Cref{lem:fourth moment}, we have
		\begin{equation*}
			\E\big[|X_N|^4\big] = 2 + \frac{2 N^2 \Delta(Q_{L_N})}{V_{Q_{L_N}}(N)^2} + O\bigg( \frac{ L_N(\log L_N)^4}{(\log N)^{4/5}} \bigg).
		\end{equation*}
		Since $L_N \to \infty$ as $N \to \infty$, we have $\Delta(Q_{L_N}) \to \Delta(g)$ and $\|Q_{L_N}\|_2 \to \|g\|_2$. Thus, taking $N \to \infty$, the fourth moment converges to:
		\begin{equation}\label{eq: limit 4th 1}
			\lim_{N \to \infty} \E\big[|X_N|^4\big] = 2 + \frac{2 \Delta(g)}{\|g\|_2^4}.
		\end{equation}
		
		On the other hand, we can evaluate the limit using our weak convergence $X_N \to^d Z$ coupled with the uniform boundedness of the sixth moment. Fix a large real number $M > 0$. We partition the expected value as:
		\begin{equation}\label{eq: split M}
			\E\big[|X_N|^4\big] = \E\big[|X_N|^4 \bm{1}_{|X_N| \le M}\big] + \E\big[|X_N|^4 \bm{1}_{|X_N| > M}\big].
		\end{equation}
		For the tail term, we use the uniform bound \eqref{eq: bound 6} to deduce:
		\begin{equation*}
			\E\big[|X_N|^4 \bm{1}_{|X_N| > M}\big] \le \E\bigg[|X_N|^4 \cdot \frac{|X_N|^2}{M^2} \bm{1}_{|X_N| > M}\bigg] \le \frac{1}{M^2}\E\big[|X_N|^6\big] \le \frac{C}{M^2}.
		\end{equation*}
		For the bounded term, since $x \mapsto |x|^4 \bm{1}_{|x| \le M}$ is continuous almost everywhere with respect to the distribution of $Z$ (which is absolutely continuous), the weak convergence $X_N \to^d Z$ guarantees that:
		\begin{equation*}
			\lim_{N \to \infty} \E\big[|X_N|^4 \bm{1}_{|X_N| \le M}\big] = \E\big[|Z|^4 \bm{1}_{|Z| \le M}\big].
		\end{equation*}
		Taking the $\limsup$ and $\liminf$ of \eqref{eq: split M} as $N \to \infty$ thus yields:
		\begin{align*}
			\limsup_{N \to \infty} \E\big[|X_N|^4\big] &\le \E\big[|Z|^4 \bm{1}_{|Z| \le M}\big] + \frac{C}{M^2} \le \E\big[|Z|^4\big] + \frac{C}{M^2} = 2 + \frac{C}{M^2}, \\
			\liminf_{N \to \infty} \E\big[|X_N|^4\big] &\ge \E\big[|Z|^4 \bm{1}_{|Z| \le M}\big] \ge 2 - \frac{C}{M^2}.
		\end{align*}
		Because these bounds hold for any arbitrarily large $M > 0$, letting $M \to \infty$ forces the limit to exist and equal exactly $2$. 
		
		Comparing this to \eqref{eq: limit 4th 1}, we must have:
		$$2 + \frac{2 \Delta(g)}{\|g\|_2^4} = 2 \implies \Delta(g) = 0,$$
		which contradicts our assumption. We conclude that if $S_g(N)$ converges to a complex Gaussian, we must have $\Delta(g) = 0$.
	\end{proof}

	%------
	% Insert the bibliography.
	%------

\end{document}